\documentclass[thmsa,12pt]{article}
\normalfont\sffamily
\usepackage{amsfonts}
\usepackage{amssymb}
\usepackage[active]{srcltx}
\usepackage{color}
\usepackage[x11names]{xcolor}
\usepackage{amsmath}
\usepackage{amsthm}
\usepackage{bm} 
\usepackage{graphicx}
\usepackage{ifthen}
\usepackage{xcolor}
\usepackage{textcomp}
\usepackage{a4wide}
\usepackage{enumitem}
\usepackage{multirow}
\usepackage{diagbox}
\usepackage{color}
\usepackage[format=hang]{caption}
\usepackage{subcaption}
\usepackage{url}
\usepackage{xfrac}

\usepackage[normalem]{ulem}
\usepackage{comment}
\usepackage{pxfonts}
\usepackage[margin=2cm]{geometry} 
\usepackage{titlesec}

\usepackage{setspace}
\usepackage{footmisc}
\usepackage{accents}
\usepackage{appendix}

\usepackage{array,booktabs,ragged2e}
\newcolumntype{R}[1]{>{\RaggedRight}p{#1}}

\usepackage{cellspace}
\usepackage{hyperref}
\usepackage{hyperxmp}
\hypersetup{
  pdftitle={Optimal self-orthogonal binary linear codes},
  pdfauthor={Sascha Kurz},
  pdfsubject={We complete the determination of the minimum possible length of eight-dimensional binary self-orthogonal codes with given minimum distance. For dimension nine we give exact values and bounds for minimum distance at most $64$. By exploiting the geometric reformulation of linear codes as multisets of points, we fully redo the classification for dimension at most seven in a simplified and computer-free manner.
As a by-product we prove the non-existence of $[93,9,44]_2$ codes and show several divisibility results for putative optimal nine-dimensional binary codes.},
  pdfkeywords={self-orthogonal codes, linear codes, Griesmer bound, Galois geometry, residual codes},
  pdfcopyright={Creative Commons Lizenz CC BY 4.0},
  pdflicenseurl={https://creativecommons.org}
}

\usepackage{fancyhdr}
\usepackage{footnote}
\makesavenoteenv{tabular}

\newboolean{showtodos}
\setboolean{showtodos}{false} 
\newcommand{\todo}[1]{
  \ifthenelse{\boolean{showtodos}}{
    {\begin{center}\color{red}\textbf{ToDo: } #1\end{center}}%
  }{}%
}

\newtheorem{theorem}{Theorem}
\newtheorem{lemma}{Lemma}

\newtheorem{proposition}{Proposition}
\newtheorem{corollary}{Corollary}

\newtheorem{remark}{Remark}

\def \be {\begin{equation}}
\def \ee {\end{equation}}

\newcommand{\F}{\mathbb{F}}
\newcommand{\N}{\mathbb{N}}

\newcommand{\Z}{\mathbb{Z}}

\newcommand{\cM}{\mathcal{M}}

\newcommand{\PG}{\operatorname{PG}}
\newcommand{\Res}{\operatorname{Res}}

\newcommand{\wt}{\operatorname{wt}}
\newcommand{\supp}{\operatorname{supp}}

\usepackage[numbers]{natbib}
\begin{document}

\setstretch{1.1}

\title{Optimal self-orthogonal binary linear codes}
\author{Sascha Kurz\thanks{I am grateful to Hirotaka Shimizu for several remarks on an earlier draft.}\\ \normalsize
	Dept.\ of Mathematics, University of Bayreuth, 95440 Bayreuth, Germany\\email: sascha.kurz@uni-bayreuth.de}
\date{November 17, 2026} 

\maketitle

\begin{abstract}\noindent
We complete the determination of the minimum possible length of eight-dimensional binary self-orthogonal codes with given minimum distance. For dimension nine we give exact values and bounds for minimum distance at most $64$. By exploiting the geometric reformulation of linear codes as multisets of points, we fully redo the classification for dimension at most seven in a simplified and computer-free manner.
As a by-product we prove the non-existence of $[93,9,44]_2$ codes and show several divisibility results for putative optimal nine-dimensional binary codes.

\medskip

\end{abstract}

\setstretch{1.26} 



\pagenumbering{arabic}
\renewcommand{\headrulewidth}{0pt}
\renewcommand{\footrulewidth}{0pt}


\section{Introduction}
\label{sec:Introduction1}
A linear code is called self-orthogonal if it is contained in its dual code w.r.t.\ the standard inner product. They have been studied since the theory of error-correcting codes transitioned from an engineering tool into a mathematical discipline, see e.g.\ \cite{pless1972classification}. There are e.g.\ connections to lattices~\cite{wan2001characteristic}, designs \cite{assmus1969new,tonchev1989self}, groups \cite{key2007permutation,conway1999sphere}, DNA codes \cite[Section 8]{Shi2023Classification}, and quantum error-correcting codes \cite{calderbank1996good}. We call a linear $[n,k,d]_q$ code (length) optimal if there does not exist a linear $[n',k,d]_q$ code with $n'<n$. By $n_q(k,d)$ we denote the minimum possible length of an $[n,k,d]_q$ code. The function $n_2(k,\cdot)$ is completely known for all $k\le 8$ \cite{bouyukliev2000smallest}. The majority of the current knowledge on $n_q(k,d)$ can be looked up in the online table \cite{grassl_codetables}. If we additionally require that the linear code has to be self-orthogonal we use the notation $\tilde{n}_q(k,d)$ for the minimum possible length. In terms of distance optimality, $\tilde{n}_2(4,\cdot)$ was completely
determined in \cite{li2008classification}. In a sequence of papers this was extended to $\tilde{n}_2(k,\cdot)$ for all $k\le 7$, see \cite{shi2025binary} and the references cited therein. Here we completely determine $\tilde{n}_2(8,\cdot)$, cf.\ \cite{kurz2025optimal} where the optimal parameters of $8$-dimensional binary codes corresponding to additive quaternary codes were recently determined. For the binary case dimension $8$ currently seems to be a frontier since still many values of $n_2(9,d)$ are unknown. So, in Theorem~\ref{thm_SO_k_9} and Corollary~\ref{cor_SO_equals_BLC} we give a few exact values and bounds for $\tilde{n}_2(9,d)$ when $d\le 64$. See e.g.\ \cite{an2025shortest,Lim2026ConstraintPreserving,shimizu2026selforthogonal} for some very recent activity in this direction. While the problem can be formulated for arbitrary field sizes $q$, here we mainly restrict ourselves to $q=2$. For recent constructions of self-orthogonal linear codes over $\F_q$ with $q\ge 3$ see e.g.\ \cite{li2026construction,li2024construction,xiang2026ternary,xie2026optimal,zhang2026new}. Full classifications of self-orthogonal codes with small parameters have been obtained in \cite{BouyuklievOstergaard2006Binary} for $q=2$ and in \cite{bouyukliev2005classification} for $q\in \{3,4\}$. There is a well known correspondence between linear codes and multisets of points in projective spaces
\cite{dodunekov1998codes}, see e.g.\ \cite{landjev2000linear} for a survey on known results for linear codes and their geometric reformulations.
However, in the context of self-orthogonal codes we are not aware of papers extensively using these methods. So, we will do so in this paper.

The remaining part is structured as follows. In Section~\ref{sec_preliminaries} we introduce the necessary preliminaries before we determine $\tilde{n}_2(k,\cdot)$ for all $k\le 7$ in Section~\ref{sec_opt_at_most_seven}. The determination of $\tilde{n}_2(8,\cdot)$ is presented in Section~\ref{sec_opt_8}, see Theorem~\ref{main_thm}. In Section~\ref{sec_opt_9} we state and improve the current knowledge on $\tilde{n}_2(9,d)$ for $d\le 64$, see Theorem~\ref{thm_SO_k_9}. As a by-product we prove the non-existence of $[93,9,44]_2$ codes in Theorem~\ref{thm_no_93_9_44} and show several divisibility results for putative optimal nine-dimensional binary codes in Proposition~\ref{prop_opt_implies_divisibility}. Generator matrices of examples found by computer searches are listed in Appendix~\ref{sec_generator_matrices_of_SO_codes}.

\section{Preliminaries}
\label{sec_preliminaries}
In this section we introduce the basic notation and first auxiliary results. We start with the coding theoretic formulation in Subsection~\ref{subsec_coding_theoretic_notation}. The well-known reformulation of linear codes as multisets of points in projective geometries, see e.g.\ \cite[{\S}1.1.2]{tsfasman1991agcodes} or \cite{dodunekov1998codes}, is briefly outlined in Subsection~\ref{subsec_geometric_notation}. For the reader's sake we
provide proofs whenever they are short, instructive, or we did not find
our precise formulation in the literature, see e.g.\ the characterization
results for binary self-orthogonal codes in Proposition~\ref{prop_SO_characterization} and Lemma~\ref{lemma_geometric_SO}. However, we claim no priority at all.
It is well known that the minimum possible length of a linear code with given dimension, field size, and minimum distance $d$ is given by the so-called Griesmer bound, see Inequality~(\ref{eq_griesmer_bound}) for all sufficiently large $d$ and attained by the Solomon--Stiffler construction \cite{griesmer1960bound,solomon1965algebraically}.
If we additionally require the linear codes to be self-orthogonal, we have to slightly modify the Griesmer bound, see Inequality~(\ref{ie_Griesmer_SO}), and the attaining construction, see Subsection~\ref{subsec_geometric_notation}. In order to make the paper self-contained we provide more technical details on the Griesmer bound and the Solomon--Stiffler construction in Subsection~\ref{subsec_Solomon_Stiffler}.

\subsection{Coding theoretic notation}
\label{subsec_coding_theoretic_notation}
Let $\F_q$ denote the field with $q$ elements and $[n]:=\{1,\dots,n\}$. A \emph{linear} $[n,k]_q$ code $C$ is a $k$-dimensional subspace of $\F_q^n$.
If the rowspace of $G\in\F_q^{k\times n}$ equals $C$, we speak of a \emph{generator matrix}. The parameter $n$ is called the \emph{length} and the parameter $k$ is
called the \emph{dimension} of $C$. Elements of $C$ are called \emph{codewords}. For any $c=\left(c_1,\dots,c_n\right)\in\F_q^n$ the
\emph{support} of $c$ is defined as
$
   \supp(c):=\left\{i\in[n]\,:\,c_i\neq 0\right\}
$
and the size of the support $|\supp(c)|$ is called the
\emph{Hamming weight} $\wt(c)$ of $c$.
Multiplying all entries at some position $1\le i\le n$ by some $\lambda\in\F_q^*:=\F_q\backslash\{0\}$ does not change the weights of the codewords, which is also true for permutations of the columns. In those cases we call the corresponding codes \emph{equivalent}.\footnote{For non-prime fields $\F_q$ also the application of a~field automorphism to the elements in all coordinate positions yields codes that are equivalent in the same sense.} Having columns of generator matrices in mind, we also call $g\in \F_q^k$ and $\lambda g$ for
some $\lambda\in\F_q^*$ equivalent.
The \emph{minimum (Hamming) distance} of a linear code $C$ is defined as the smallest non-zero Hamming weight $\wt(c)$ over all codewords $c\in C$ and the support of $C$ is given by
$\supp(C)=\cup_{c\in C} \supp(c)$. We call $|\supp(C)|$ the \emph{effective length} and of \emph{full length} if $|\supp(C)|=n$. If $C$ is an $[n,k]_q$ code with minimum distance at least $d$\footnote{We use this convention to avoid the notation of an $[n,k,\ge d]_q$ code at several places. If needed in this paper, the existence of a codeword of weight $d$ will always follow from $n<n_q(k,d+1)$.}, we speak of an $[n,k,d]_q$ code. By $n_q(k,d)$ we denote the minimum possible length of an $[n,k, d]_q$ code, which is lower bounded by the so-called \emph{Griesmer bound}
\cite{griesmer1960bound,solomon1965algebraically}
\begin{equation}
  \label{eq_griesmer_bound}
  n_q(k,d)\ge \sum_{i=0}^{k-1} \left\lceil\frac{d}{q^i}\right\rceil=:g_q(k,d).
\end{equation}
If $n=g_q(k,d)$, then we speak of a \emph{Griesmer code}. They exist for all sufficiently large $d$ \cite{solomon1965algebraically}, so that $n_q(k,d)=g_q(k,d)$ and the determination of $n_q(k,d)$ as a function of $d$ becomes a finite problem.
For the current state of knowledge on bounds for $n_q(k,d)$ we refer to the online table~\cite{grassl_codetables}.
Interpreted as columns of a generator matrix, a $k$-dimensional \emph{simplex code} $\mathcal{S}_q(k)$ over $\F_q$ is given by the $[k]_q:=\tfrac{q^k-1}{q-1}$ non-equivalent vectors in $\F_q^k\backslash\{0\}$ and the \emph{first order Reed--Muller code} $\mathcal{RM}_q(1,k)$ is given by the $q^k$ vectors in $\F_q^{k+1}$ starting with a $1$. These are $\left[[k]_q,k,q^{k-1}\right]_q$ and $\left[q^k,k+1,q^k-q^{k-1}\right]_q$ codes, respectively, see e.g.\ \cite{huffman2003fundamentals}.
A linear code is called \emph{$\Delta$-divisible} if the weights of all codewords are divisible by $\Delta$. E.g., $\mathcal{S}_q(k)$ and $\mathcal{RM}_q(1,k)$ are Griesmer codes and $q^{k-1}$-divisible for all $k\ge 1$.
\begin{theorem}\label{thm_griesmer_divisible}
  (\cite[Theorem 1]{Ward1998}, \cite{landjev2001geometric}, or \cite{Dodunekov1984Minimum} for $q=2$)
  If $C$ is an $[n,k,d]_p$ Griesmer code, where $p$ is a prime, and $p^r$ divides $d$, then $C$ is $p^r$-divisible.
\end{theorem}
\noindent
For the divisibility of Griesmer codes when the field size $q$ is a prime power see  e.g.\ \cite{deng2026divisibility,deng2026Ward,zhang2026hyperplane}. In Lemma~\ref{lemma_divisible_row_span} we will show a divisibility criterion for binary non-Griesmer codes.
For each $S\subseteq[n]$ by $C_S$ we
denote the linear code consisting of all elements of an $[n,k]_q$ code $C$ restricted to the coordinates in $S$. With this the \emph{residual code}
$\Res(C;c)$ of $C$ w.r.t.\ a codeword $c\in C$ is given
by $C_{[n]\backslash\supp(c)}$ and has length $n-\wt(c)$ as well as a dimension of at most $k-1$. If $\wt(c)<\frac{dq}{q-1}$ the dimension is exactly $k-1$. Unless we explicitly say otherwise, here a residual code is always meant w.r.t.\ a non-zero codeword.
\begin{lemma} (E.g.~\cite{dodunekov1985optimal,dodunekov1984minimal})
  \label{lemma_residual_code_bound}
  If an $[n,k,d]_q$ code $C$ contains a codeword $c$ of weight $w$ with $w<\frac{dq}{q-1}$, then the residual code of
  $C$ with respect to $c$ is an $\left[n-w,k-1,d-w+\left\lceil \tfrac{w}{q}\right\rceil\right]_q$-code. The residual code has full length iff $C$ has.
\end{lemma}
\noindent
Some authors speak of the \emph{one-step-Griesmer bound}. The exclusion of codeword weights $w$ via $n-w<n_q\!\left(k-1,d-w+\left\lceil \tfrac{w}{q}\right\rceil\right)$ and Lemma~\ref{lemma_residual_code_bound} is called \emph{residual code argument}. The residual code of a Griesmer code w.r.t.\ a minimum weight codeword is a Griesmer code again, see e.g.\ \cite[Lemma 29]{bouyukliev2009classification}, \cite{deng2026generalized}, or Lemma~\ref{lemma_griesmer_observations}. Applying Lemma~\ref{lemma_residual_code_bound} with minimum weight codewords recursively yields the Griesmer bound (\ref{eq_griesmer_bound}). The residual code of a codeword in a $q^r$-divisible $[n,k,d]_q$ code  is $q^{r-1}$-divisible if $k,r\ge 1$, see e.g.\ \cite{kurz2021divisible,ward2001divisible} for surveys on divisible codes.
For $x,y\in \F_q^n$ the \emph{scalar product} is given by $\langle x,y\rangle=\sum_{i=1}^n x_iy_i$. With this the \emph{dual} of an
$[n,k]_q$ code $C$ is given by the $[n,n-k]_q$ code
$
  C^\perp=\left\{y\in\F_q^n\,:\, \langle x,y\rangle=0\,\forall x\in C\right\}$.
The minimum distance of the dual code is called the \emph{dual minimum distance} and denoted by $d^\perp$. We call a linear code \emph{projective} if $d^\perp\ge 3$.
If $C\subseteq C^\perp$ then $C$ is called \emph{self-orthogonal} (SO) and \emph{self-dual} if $C=C^\perp$. So, an $[n,k]_q$ code $C$ is self-orthogonal iff $\langle c,c'\rangle=0$ for all (possibly equal) pairs of (non-zero) codewords $c,c'\in C$. By $\tilde{n}_q(k,d)$ we denote the minimum possible length of an $[n,k,d]_q$ SO code. Clearly, we have $\tilde{n}_q(k,d)\ge n_q(k,d)$, which we call the \emph{coding lower bound}.

\medskip

Here we will mainly consider \emph{binary codes}, i.e.\ $q=2$.\footnote{Note that e.g.\ self-orthogonality is not preserved in general by monomial transforms for $q>3$, see e.g.\ the generator matrices $\begin{pmatrix} 1 & 2 & 0 & 0\\ 0&0&1&2\end{pmatrix}$ and $\begin{pmatrix} 1 & 4 & 0 & 0\\ 0&0&1&2\end{pmatrix}$ over $\F_5$. So, there e.g.\ does not exist an equivalent geometric reformulation in terms of multisets of points, see Subsection~\ref{subsec_geometric_notation}.} Let
$C$ be an $[n,k,d]_2$ code. We speak of \emph{even},
\emph{doubly-even}, and \emph{triply-even} codes if $C$ is $2$-, $4$-,
or $8$-divisible, respectively. A code that is even but not doubly-even
is called \emph{singly-even}. Moreover, we have
\begin{equation}
  \label{eq_weight_sum}
  \wt(c+c')=\wt(c)+\wt(c')-2\cdot|\supp(c)\cap\supp(c')|
\end{equation}
for each $c,c'\in C$.
Using the \emph{Schur product} (Hadamard product) $c\star c'$,
i.e.\ componentwise multiplication, we have
\begin{equation}
  |\supp(c)\cap\supp(c')|=\wt(c\star c').
\end{equation}
If $c,c'\in C\backslash\{0\}$ with $c\neq c'$  and $c'_r$ denotes the codeword corresponding to $c'$ in $\operatorname{Res}(C;c)$, then we have
\begin{equation}
  \label{eq_weight_residual}
  \wt(c'_r)=\wt(c')-|\supp(c)\cap\supp(c')|=\wt(c')-\wt(c\star c').
\end{equation}
Let us now characterize binary SO codes. Clearly, an $[n,1]_2$ code is SO iff it is even.
\begin{proposition}
  \label{prop_SO_characterization}
 Let $C$ be an $[n,k]_2$ code, where $k\ge 2$. Then, the following statements are equivalent:
 \begin{enumerate}[noitemsep]
  \item[(A)] $C$ is SO.
  \item[(B)] $C$ is even and the doubly-even codewords form a subcode of dimension $k-1$ or $k$.
  \item[(C)] All residual codes are even.
  \item[(D)] The rows of a generator matrix $G$ have even weights and are pairwise orthogonal.
 \end{enumerate}
\end{proposition}
\begin{proof}
  Let us first assume that $C$ is SO, i.e.\ property (A). Since $\langle c,c\rangle=0$ for every codeword we have $\wt(c)\equiv 0\pmod 2$, i.e.\ $C$ is even.
  For two different codewords $\langle c,c'\rangle=0$ implies
  $|\supp(c)\cap\supp(c')|=\wt(c\star c')\equiv 0\pmod 2$. Using Equation~(\ref{eq_weight_sum}) we conclude that the sum of two doubly-even codewords and the sum of two singly-even codewords is always doubly-even, i.e.\ the set of doubly-even codewords forms a subcode of dimension at least $k-1$. For a codeword $c'_r\in\operatorname{Res}(C;c)$ corresponding to $c'\in C$ Equation~(\ref{eq_weight_residual}) gives
  $\wt(c_r')\equiv 0\pmod 2$, i.e., the residual code is even.

  Now let us assume that $C$ satisfies (B). Since $C$ is even we have $\langle c,c\rangle=0$. For each pair of different codewords $c,c'\in C$
  property (B) implies that the number of singly-even codewords of $\{c,c',c'+c\}$ is either zero or two. Thus, Equation~(\ref{eq_weight_sum})
  implies $|\supp(c)\cap\supp(c')|=\wt(c\star c')\equiv 0\pmod 2$, i.e.\
  $\langle c,c'\rangle=0$.

  Let us now assume that $C$ satisfies (C). Due to Equation~(\ref{eq_weight_residual}) it suffices to show that $C$ is even in order to conclude that $C$ is SO. From Equation~(\ref{eq_weight_residual}) we conclude $\wt(c')\equiv \wt(c\star c')\equiv \wt(c)\pmod 2$ for all pairs of different codewords $c,c'\in C\backslash\{0\}$, i.e.\ the non-zero weights are either all even or all
  odd. The latter case is impossible for $k\ge 2$ due to Equation~(\ref{eq_weight_sum}).

  The implication (D) $\Rightarrow$ (A) follows from the bilinearity of the scalar product. The other direction is obvious.
\end{proof}

\smallskip

\noindent
Property (B) is a refinement of the commonly stated implications that each SO binary code is even and that each doubly-even binary code is SO. Similarly, for each binary code Equation~(\ref{eq_weight_sum}) implies that the set of even codewords forms a subcode of dimension at least $k-1$ in each binary $[n,k]_2$ code. The implication (A) $\Rightarrow$ (B) is e.g.\ mentioned in \cite[Section II]{conway1990new}. For the other direction see e.g.\ \cite[Proposition 3.2]{li2025constributionsbinaryselforthogonal}. Lemma~\ref{lemma_residual_code_bound} and property (C) directly imply the following improvement of Lemma~\ref{lemma_residual_code_bound} for binary SO codes.
\begin{lemma}
  \label{lemma_residual_SO_code_bound}
  If an $[n,k,d]_2$ SO code $C$ with $k\ge 2$ contains a codeword $c$ of weight $w$ with $w<2d$ and $d-\tfrac{w}{2}\equiv 1\pmod 2$, then the residual code of
  $C$ with respect to $c$ is an $\left[n-w,k-1,d-\tfrac{w}{2}+1\right]_2$-code.
\end{lemma}
\noindent
For $w=d$ Lemma~\ref{lemma_residual_SO_code_bound} was shown in \cite[Lemma 3.1]{li2008classification} and \cite[Lemma 3.7]{shi2025binary}. Since the residual of a Griesmer $[n,k,d]_2$ code w.r.t.\ a minimum weight codeword
is a Griesmer $[n-d,k-1,\tfrac{d}{2}]_2$ code, Lemma~\ref{lemma_residual_SO_code_bound} implies
\begin{equation}
  \label{ie_Griesmer_SO}
  \tilde{n}_2(k,d)\ge \left\{ \begin{array}{rcl} g_2(k,d')&:& d'\equiv 0\pmod 4,\\g_2(k,d')+1&:& d'\equiv 2\pmod 4\end{array}\right.
  =:\tilde{g}_2(k,d),
\end{equation}
where $d':=2\left\lceil \tfrac{d}{2}\right\rceil$ and we assume $k\ge 2$. We call Inequality~(\ref{ie_Griesmer_SO}) the \emph{Griesmer SO bound} and remark that it is attained with equality for all sufficiently large $d$,\footnote{As mentioned before, the same property  holds for linear codes over $\F_q$ and the ordinary Griesmer bound. Other instances of Griesmer-type bounds with this strong asymptotic equivalence are given by additive codes \cite{Kurz2024Additive} and linear codes w.r.t.\ the $b$-symbol metric \cite{kurz2025linear}.} see
\cite{shi2025binary} or Subsection~\ref{subsec_geometric_notation}. Thus, the determination of $\tilde{n}_2(k,\cdot)$ becomes a finite problem for each dimension $k$.
The exclusion of codeword weights $w$ via $n-w<n_2\!\left(k-1,d-\tfrac{w}{2}+1\right)$ and Lemma~\ref{lemma_residual_SO_code_bound} is called \emph{residual SO code argument}. Since the residual codes are even, we can deduce further exclusion criteria for large weights.
\begin{lemma}
  \label{lemma_residual_SO_large_weight}
  Let $C$ be an $[n,k,d]_2$ SO code with $k\ge 2$. If $C$ is full-length, then there cannot be a codeword of weight $n-1$. If $C$ is projective, then there cannot be a codeword of weight $n-2$.
\end{lemma}

\begin{proposition} (\cite{macwilliams1977theory}, MacWilliams Identities)
  \label{prop_MacWilliams_identitites}
   Let $C$ be an $[n,k,d]_2$ code and $C^\perp$ be the dual code of $C$. Let $A_i(C)$ and $B_i(C)$ be the number of codewords of weight $i$ in $C$ and
   $C^\perp$, respectively. With this, we have
   \begin{equation}
     \label{eq_MacWilliams}
     \sum_{j=0}^n K_i(j)A_j(C)=2^kB_i(C),\quad 0\le i\le n
   \end{equation}
   where
   $$
     K_i(j)=\sum_{s=0}^n (-1)^s {{n-j}\choose{i-s}}{{j}\choose{s}},\quad 0\le i\le n
   $$
   are the binary Krawtchouk polynomials. We will simplify the notation to $A_i$ and $B_i$ whenever $C$ is clear from the context.
 \end{proposition}
 We call the sequence $\left(i^{A_i}\right)_{i>0}$, where we leave out entries with a zero exponent, the \emph{weight distribution} of $C$.
 Whenever we speak of the first $l$ MacWilliams identities, we mean Equation~(\ref{eq_MacWilliams}) for $0\le i\le l-1$.
 Adding the non-negativity constraints $A_i,B_i\ge 0$ and $A_i=0$ for forbidden weights $i$, we obtain a linear program where we can maximize or minimize certain quantities, which is called the \emph{linear programming method for linear codes}.
 If there is no feasible solution, then there is no linear code with the specified parameters. Adding additional equations or inequalities strengthens the formulation. E.g.\ we can add integer rounding cuts.
 For an $[n,k,d]_2$ code of full length, i.e.\ $B_1=0$, we can rewrite the first four MacWilliams identities to
 \begin{eqnarray}
   \sum_{i>0} A_i   &=& 2^k-1,\label{eq_mw1}\\
   \sum_{i} iA_i    &=& 2^{k-1}n, \label{eq_mw2}\\
   \sum_{i} i^2 A_i &=& 2^{k-1}\cdot\left(B_2+n(n+1)/2\right),\label{eq_mw3}\\
   \sum_{i} i^3 A_i &=& 2^{k-2}\cdot\left(3(B_2n-B_3)+n^2(n+3)/2\right).\label{eq_mw4}
 \end{eqnarray}
 \begin{proposition}(\cite[Proposition 5]{dodunekov1999some}, cf.~\cite{simonis1994restrictions})
  \label{prop_div_one_more}
  Let $C$ be an $[n,k,d]$-code with all weights divisible by $\Delta:=2^a$ and let $\left(A_i\right)_{i=0,1,\dots,n}$ be the weight distribution of $C$. Put
  \begin{eqnarray*}
    \alpha&:=&\min\{k-a-1,a+1\},\\
    \beta&:=&\lfloor(k-a+1)/2\rfloor,\text{ and}\\
    \delta&:=&\min\{2\Delta i\,\mid\,A_{2\Delta i }\neq 0\wedge i>0\}.
  \end{eqnarray*}
  Then the integer
  $$
    T:=\sum_{i=0}^{\lfloor n/(2\Delta)\rfloor} A_{2\Delta i}
  $$
  satisfies the following conditions.
  \begin{enumerate}[noitemsep]
    \item \label{div_one_more_case1}
          $T$ is divisible by $2^{\lfloor(k-1)/(a+1)\rfloor}$.
    \item \label{div_one_more_case2}
          If $T<2^{k-a}$, then
          $
            T=2^{k-a}-2^{k-a-t}
          $
          for some integer $t$ satisfying $1\le t\le \max\{\alpha,\beta\}$. Moreover, if $t>\beta$, then $C$ has an $[n,k-a-2,\delta]$-subcode and if $t\le \beta$, it has an
          $[n,k-a-t,\delta]$-subcode.
    \item \label{div_one_more_case3}
          If $T>2^k-2^{k-a}$, then
          $
            T=2^k-2^{k-a}+2^{k-a-t}
          $
          for some integer $t$ satisfying $0\le t\le \max\{\alpha,\beta\}$. Moreover, if $a=1$, then $C$ has an $[n,k-t,\delta]$-subcode. If $a>1$, then $C$ has an
          $[n,k-1,\delta]$-subcode unless $t=a+1\le k-a-1$, in which case it has an $[n,k-2,\delta]$-subcode.
  \end{enumerate}
\end{proposition}
A special and well-known subcase is that the number of even weight codewords in an $[n,k]_2$ code is either $2^{k-1}$ or $2^k$.

 The weight enumerator can be generalized to the split weight enumerator based on a partition of the coordinates \cite{simonis1995macwilliams}. For a linear
 programming method based on the split weight enumerator we refer e.g.\ to \cite{jaffe1997brief}. Here we will apply the split weight enumerator for a partition given by the support of some codeword and its complement.

\subsection{Geometric notation}
\label{subsec_geometric_notation}

The set of all subspaces of $\F_q^r$, ordered by the incidence relation
$\subseteq$, is called \emph{$(r-1)$-dimensional projective geometry over
$\F_q$} and denoted by $\PG(r-1,q)$. Employing this algebraic notion
of dimension instead of the geometric one, we will use the term $i$-space
to denote an $i$-dimensional subspace of $\F_q^r$. To highlight the
important geometric interpretation of subspaces we will call $1$-, $2$-,
and $(r-1)$-spaces points, lines, and hyperplanes, respectively. For two
subspaces $S$ and $S'$ we write $S\le S'$ if $S$ is contained in $S'$. Moreover,
we say that $S$ and $S'$ are \emph{incident} iff $S\le S'$ or $S\ge S'$.
Let $[i]_q:=\tfrac{q^i-1}{q-1}$ denote the number of points
of an arbitrary $i$-space in $\PG(r-1,q)$ where $r\ge i$.
A \emph{premultiset of points} $\cM$ of $\PG(r-1,q)$ is a mapping from the
set of points of $\PG(r-1,q)$ to $\Z$. The value $\cM(P)$ is called the
\emph{multiplicity} of $P$. We extend this notation additively to each subspace
$S$ via $\cM(S)=\sum_{P\le S} \cM(P)$. So, the cardinality of $\cM$ equals
$\cM(V)$, where $V$ denotes the ambient space. For a hyperplane $H$
the number of points of $\cM$ in $H$ is given by $\cM(H)$. If $\cM(P)\in \N$
for each point $P$, we say that $\cM$ is a \emph{multiset of points}. For each
subspace $S$ we denote its \emph{characteristic function} by $\chi_S$,
i.e.\ $\chi_S(P)=1$ if $P\le S$ and $\chi_S(P)=0$ otherwise.

Starting from a generator matrix $G\in\F_q^{k\times n}$ of an $[n,k,d]_q$ code $C$ we can associate a multiset of points $\cM$ in $\PG(k-1,q)$ by considering the spans of the columns of $G$. The cardinality of $\cM$ is $\#\cM=n$ iff $C$ is of full length. For the other direction we can build up a generator matrix $G$ of $C$ by taking $\cM(P)$ copies of a representing vector for every point $P$. We call $\cM$ in $\PG(r-1,q)$ \emph{spanning}
if the points with positive multiplicity, called the \emph{support}, span the entire ambient space. The dimension of the associated linear code equals $r$ iff $\cM$ is spanning.
Let $C$ be a full length $[n,k]_q$ code and $\cM$ in $\PG(k-1,q)$ its associated spanning multiset of points. We remark that the ambiguities in these constructions result in equivalent linear codes or multisets of points, see e.g.\ \cite[{\S}1.1.2]{tsfasman1991agcodes} or \cite{dodunekov1998codes} for more details. As
an example we mention that $\mathcal{S}_q(k)$ and $\mathcal{RM}_q(1,k-1)$ correspond to $\chi_V$ and $\chi_V-\chi_H$, where $V$ denotes the ambient space $\PG(k-1,q)$ and $H$ is a hyperplane.
Each non-zero codeword $c$ corresponds to a hyperplane $H$ and we have $\wt(c)=n-\cM(H)$. For the other direction, each hyperplane corresponds to the $q-1$ multiples of a non-zero codeword. A multiset of points $\cM$ is called \emph{$m$-arc} if $\cM(H)\le m$ for every hyperplane $H$. I.e.\ a full-length $[n,k,d]_q$ code corresponds to a spanning $(n-d)$-arc. Similarly, a multiset of points $\cM$ is called \emph{$m$-minihyper} (or \emph{$m$-fold blocking set w.r.t.~hyperplanes}) if $\cM(H)\ge m$ for every hyperplane $H$. If $\cM$ is a spanning $m$-arc and $\cM'$ is an $m'$-minihyper in $\PG(k-1,q)$, then $\cM-\cM'$ is a spanning $(m-m')$-arc iff it is a spanning multiset of points. In coding theoretic terms the deletion of $\cM'$ is called \emph{puncturing}.
\todo{Study the known examples from the literature where puncturing was used in order to find interesting candidates for $\cM'$
  and possible geometric constructions.}
For a multiset of points $\cM$ we denote the maximum point multiplicity by  $\gamma(\cM)$, so that e.g.\ $\gamma(\cM)=1$ iff the corresponding linear code $C$ is projective. If $K$ is some subspace not equal to the ambient space $V$, then the \emph{projection $\cM_K$ of $\cM$ through $K$} is the multiset of points in the factor geometry $V/K$ with multiplicity $\cM_K(S/K)=\cM(S)-\cM(K)$ for every subspace $K\subseteq S\subseteq V$. The dimension of the ambient space drops by $\dim(K)$ and the cardinality drops by $\cM(K)$, while the $m$-arc property is preserved. In coding theory terms $\cM_K$ corresponds to a codimension $\dim(K)$ subcode of $C$ when $\cM$ corresponds to $C$ and projection is called \emph{truncation} or \emph{shortening}. If $\cM,\cM'$ are spanning $m$- and spanning $m'$-arcs in $\PG(k-1,q)$, respectively, then $\cM+\cM'$ is a spanning $(m+m')$-arc in $\PG(k-1,q)$. If $\cM'$ spans a $(k-1)$-space we can choose a hyperplane $H$ for the embedding, so that
$\cM+\cM'$ is a spanning $(\le \max\{m+m',\cM(H)+\#\cM'\})$-arc in $\PG(k-1,q)$. In terms of corresponding generator matrices $G,G'$ this corresponds to the juxtaposition $\left(G|\tilde{G}'\right)$ where $\tilde{G}'$ is obtained from  $G'$ by suitably padding one row. An example
is given by the unique $[38,6,18]_2$ Griesmer code that can be obtained in this way from a $[32,6,16]_2$ and a $[6,5,2]_2$ code.

We call a multiset of points $\cM$ $\Delta$-divisible iff its corresponding linear code is, i.e.\ if we have $\cM(H)\equiv \#\cM\pmod \Delta$ for every hyperplane $H$. Similarly, we call a  multiset of points in $\PG(k-1,2)$ self-orthogonal (SO) if its corresponding linear code is. Clearly, projection preserves divisibility and self-orthogonality.
\begin{lemma}
  \label{lemma_geometric_SO}
  Let $\cM$ be a spanning multiset of points in $\PG(k-1,2)$, where $k\ge 2$. Then, $\cM$ is SO iff $\cM(H)\equiv\cM(K)\equiv\#\cM\pmod 2$ for every hyperplane $H$ and every subspace $K$ of codimension $2$.
\end{lemma}
\begin{proof}
  By $C$ denote the corresponding $[n,k]_2$ code. Let $c$ denote an arbitrary non-zero codeword and $H$ the corresponding hyperplane. The
  statements $\wt(c)\equiv 0\pmod 2$, $\cM(H)\equiv \#\cM\pmod{2}$, and
  $\langle c,c\rangle=0$ are pairwise equivalent. Now let $c,c'
  $ be two different non-zero codewords and $H,H'$ be the corresponding hyperplanes. We have $\langle c,c'\rangle=0$ iff $|\supp(c)\cap \supp(c')|\equiv 0\pmod 2$. Setting $K:=H\cap H'$ we observe
  \begin{eqnarray*}
    |\supp(c)\cap \supp(c')|&=&\#\cM-\cM(H\cup H')
    =\#\cM-\cM(H)-\cM(H')+\cM(K)\\
    &=&(\#\cM\!-\!\cM(H))+(\#\cM\!-\!\cM(H'))-(\#\cM\!-\!\cM(K))\equiv
    \#\cM\!-\!\cM(K)\pmod 2.
  \end{eqnarray*}
\end{proof}

\begin{corollary}
  \label{cor_lin_comb_SO}
  If $\cM_1$ is an arbitrary multiset of points and $\cM_2,\cM_3$ are SO multisets of points in $\PG(k-1,2)$, then $2\lambda\cdot\cM_1$, $\cM_2+\cM_3$, and $\lambda\cdot \cM_2$ are SO multisets of points for all $\lambda\in\N_{\ge 1}$. If $\cM_2-\cM_3$ is a multiset of points, then it is also SO.
\end{corollary}
So, $\mathcal{S}_2(k)$ and $\mathcal{RM}_2(1,k)$ are SO iff $k\ge 3$.
\begin{corollary}
  \label{cor_sum_SO_shift}
  Let $\cM=\cM_1+\cM_2$ be the sum of two multisets of points in $\PG(k-1,2)$, where $\cM_1$ is SO. Then, $\cM$ is SO iff $\cM_2$ is SO.
\end{corollary}
In coding theory terms, Corollary~\ref{cor_sum_SO_shift} was shown in \cite[Lemma 3.1]{shi2023two} for the special case where $\cM_1$ in $\PG(k-1,2)$, where $k\ge 3$, has constant point multiplicity. In \cite[Lemma 3.1]{shi2025binary} the extra assumption on $\cM_1$ was removed. Corollary~\ref{cor_lin_comb_SO} directly implies the following construction, cf.~\cite{shi2025binary}.
\begin{lemma}
  \label{lemma_minus_double_point}
  Let $\cM$ be a spanning SO multiset of points in $\PG(k-1,2)$ such that $\cM(H)\le \#\cM -d$ for every hyperplane $H$ and $P$ some point with $\cM(P)\ge 3$. Then, $\cM-2\chi_P$ corresponds to an $[\#\cM-2,k,d-2]_2$ SO code.\footnote{Note that $\cM(P)\ge 2$ is sufficient if $d>2$.}
\end{lemma}
With this, we can easily show that the Griesmer SO bound in Inequality~(\ref{ie_Griesmer_SO}) can always be attained with equality if $d$ is sufficiently large. For $d\equiv 0\pmod 4$ any $[g_2(k,d),k,d]_2$ Griesmer code is $4$-divisible by Theorem~\ref{thm_griesmer_divisible}, which implies SO. For $d\equiv 2\pmod 4$ we combine the same reasoning
to $d'=d+2\equiv 0\pmod 4$ with Lemma~\ref{lemma_minus_double_point}. Since sufficiently large $d$ implies sufficiently large $n$ and the existence of a point $P$ with multiplicity at least $3$, it just remains to show $g_2(k,d')=g_2(k,d)+3$, see e.g.~\cite{shi2025binary} or Lemma~\ref{lemma_griesmer_observations} for details.

It is a nice exercise to show that a multiset of points $\cM$ in $\PG(k-1,2)$, where $k\le 4$, is SO iff it can be recursively obtained from Corollary~\ref{cor_lin_comb_SO} using just the characteristic functions of $3$- and $4$-spaces. For $k>4$ there are more SO multisets, see e.g.\ \cite[Appendix B]{Bouyukliev2006}.

The first three MacWilliams identities can be shown by double-counting in the geometric setting, see e.g.\ \cite[Section 2.2]{kurz2021divisible}. Let $C$ be a linear code, $\cM$ its corresponding multiset of points, $c\in C$, and $H$ the hyperplane corresponding to $c$. The partition induced by the support of $c$ and its complement corresponds to the restriction of $\cM$ to $H$ and its complement $\overline{H}$. Over $\F_2$, a dual codeword of weight three geometrically corresponds to a line $L$. The fact that $L$ intersects $H$ in either three or one point directly translates to a condition for the dual variables in the split weight enumerator for $c$.
\todo{Which incidences need to be double-counted in the geometric setting in order to obtain the first few constraints of the split weight enumerator?}

\subsection{Parameters of Griesmer codes and the Solomon--Stiffler construction}
\label{subsec_Solomon_Stiffler}
Here we want to briefly describe the Solomon--Stiffler construction \cite{solomon1965algebraically} from 1965 that attains the Griesmer bound for all parameters if the minimum distance is sufficiently large.\footnote{While it is widely known, there still appear papers claiming new constructions of linear codes with desired properties that fall into this class, cf.~\cite{Xie2025Optimal}.}
Using $[i]_q=\tfrac{q^i-1}{q-1}$ for the number of points of an $i$-space in $\PG(k-1,q)$ we consider the following parameterization of the Griesmer bound.
\begin{lemma}
  \label{lemma_parameters_griesmer_code}
  Let $k$ and $d$ be positive integers. Write $d$ as
  \begin{equation}
    \label{eq_griesmer_representation_min_dist}
    d=\sigma q^{k-1}-\sum_{i=1}^{k-1}\varepsilon_iq^{i-1},
  \end{equation}
  where $\sigma\in\N_0$ and the $0\le\varepsilon_i<q$ are integers for all $1\le i\le k-1$. Then, Inequality~(\ref{eq_griesmer_bound})
  is satisfied with equality iff
  \begin{equation}
    \label{eq_griesmer_representation_length}
    n=\sigma[k]_q-\sum_{i=1}^{k-1}\varepsilon_i[i]_q,
  \end{equation}
  which is equivalent to
  \begin{equation}
    \label{eq_griesmer_representation_species}
    n-d=\sigma[k-1]_q-\sum_{i=1}^{k-1}\varepsilon_i[i-1]_q.
  \end{equation}
\end{lemma}

\begin{remark}
  Given $k$ and $d$ Equation~(\ref{eq_griesmer_representation_min_dist})
  always determines $\sigma$ and the $\varepsilon_i$ uniquely. This is
  different for Equation~(\ref{eq_griesmer_representation_species}) given
  $k$ and $n-d=s$. Here it may happen that no solution with
  $0\le \varepsilon_i\le q-1$ exists. By relaxing to
  $0\le \varepsilon_i\le q$ we can ensure existence and uniqueness is
  enforced by additionally requiring $\varepsilon_j=0$ for all $j<i$
  where $\varepsilon_i=q$ for some $i$, and by additionally specifying $\varepsilon_1$. The same is true for
  Equation~(\ref{eq_griesmer_representation_length}) given $k$ and $n$.
  For more details we refer to \cite[Chapter 2]{govaerts2003classifications}
  which also gives pointers to Hamada's work on minihypers, see e.g.\ \cite{hamada1993characterization}.
\end{remark}
Lemma~\ref{lemma_parameters_griesmer_code} and Lemma~\ref{lemma_residual_code_bound} imply the following observations.
\begin{lemma}
  \label{lemma_griesmer_observations}
  Let $k\ge 2$.
  \begin{itemize}[noitemsep]
    \item[(1)] The residual code
               of an $[n,k,d]_q$ Griesmer code w.r.t.\ a codeword of weight $d$ is an $\left[n-d,k-1,\left\lceil\tfrac{d}{q}\right\rceil\right]_q$ Griesmer code.
    \item[(2)] For each positive integer $r$ we have
               $g_q\!\left(k,d+rq^{k-1}\right)=g_q(k,d)+r\cdot\frac{q^k-1}{q-1}$.
    \item[(3)] We have $d\equiv 0\pmod {q^{k-1}}$ iff $\varepsilon_i=0$
               for all $1\le i\le k-1$. Otherwise we have $d\equiv 0\pmod {q^{i-1}}$ and $d\not\equiv 0\pmod {q^{i}}$ for the smallest
               index $1\le i\le k-1$ with $\varepsilon_i\neq 0$.
    \item[(4)] For $d\equiv 0\pmod 4$ we have $g_2(k,d)= g_2(k,d-2)+3$.
  \end{itemize}
\end{lemma}
For Lemma~\ref{lemma_griesmer_observations}.(1) see e.g.\
\cite[Lemma 29]{bouyukliev2009classification}, \cite{deng2026generalized}
and for Lemma~\ref{lemma_griesmer_observations}.(2) see e.g.\ \cite[Lemma 21]{bouyukliev2009classification}, \cite{dodunekov1985optimal}. Clearly also
Lemma~\ref{lemma_griesmer_observations}.(4) can be generalized to arbitrary field sizes $q$ and different powers of $q$ for the step size.

Using the notation and restrictions from Lemma~\ref{lemma_parameters_griesmer_code}, fixing some proper subspaces $S_1,\dots,S_l$, and denoting the ambient space $\PG(k-1,q)$ by $V$, we can state a geometric variant of the Solomon--Stiffler construction. We say that $\cM=\sigma\chi_V-\sum_{j=1}^l \chi_{S_j}$ is of \emph{Solomon--Stiffler type} $[\sigma;\varepsilon_1,\dots,\varepsilon_{k-1}]$, where $l=\sum_{i=1}^{k-1}\varepsilon_i$, iff exactly $\varepsilon_i$ subspaces $S_j$ have dimension $1\le i\le k-1$. If $\cM$ is a multiset of points, i.e.\ if $\cM(P)\ge 0$ for all points $P$, then $\cM$ corresponds to a Griesmer code. An easy sufficient condition
is $\sigma\ge \sum_{i=1}^{k-1} \varepsilon_i$, where the right hand side is upper bounded by $(q-1)(k-1)$. For more refined sufficient conditions allowing a suitable selection of the $S_j$ we e.g.\ refer to \cite{belov1974construction}. Codes of Solomon--Stiffler type are not the only known class of Griesmer codes, see e.g.\ \cite{belov1974construction,chen2025griesmer,hamada1993new,helleseth1983new,hellesth1992projective,helleseth1981new}. For $q=2$ and $d\le 2^{k-1}$ Griesmer codes have to be either of Solomon--Stiffler or Belov type \cite{belov1974construction}, see \cite{Helleseth1981}. The latter case cannot occur if $d$ is doubly-even and a singly-even example is given by the unique $[38,6,18]_2$ code. So, due to the Griesmer SO bound in Inequality~(\ref{ie_Griesmer_SO}) a projective Griesmer SO code has to be of Solomon--Stiffler type cf.~\cite{shi2025binary}.
The multiset of points $\sigma\cdot \chi_V-\cM$, i.e.\ the multiset that is subtracted, is an optimal minihyper (anticode in coding theory terms). For corresponding classification results we refer  e.g.\ to \cite{hamada1993characterization,Storme2004}.

\section{Optimal binary self-orthogonal codes with dimension at most seven}
\label{sec_opt_at_most_seven}
The determination of the minimum possible length $\tilde{n}_2(k,d)$ of a self-orthogonal $k$-dimensional binary linear code with minimum distance $d$ was completed for $k\le 7$ in \cite{shi2025binary}, see e.g.\ \cite{BouyuklievOstergaard2006Binary,kim2022self,kim2021embedding,li2008classification,shi2023two} for previous results. In this section we briefly redo the entire computation and simplify some arguments. All of our lower bounds have computer-free proofs, while e.g.\ the non-existence of $[41,7,18]_2$, $[56,7,26]_2$, $[63,7,30]_2$, and $[103,7,50]_2$ SO codes was shown in \cite{shi2025binary} using exhaustive enumeration via the software package \texttt{Q-EXTENSION} \cite{bouyukliev2007qextension}. We also state general constructions that cover cases that previously relied on computer searches.

First we summarize some easy improvements of $\tilde{n}_2(k,d)\ge \tilde{g}_2(k,d)$.
\begin{lemma}
  \label{lemma_non_existence_summary}
  Let $d$ be even and $k\ge 2$.
  \begin{itemize}[noitemsep]
   \item[(1)] If $d\equiv 2\pmod 4$ and $n<n_2(k-1,d+2)$, then no $[n,k,d]_2 $ SO code exists.
   \item[(2)] For integers $2\le t\le k-1$ and $d=2^t-2$ we have $\tilde{n}_2(k,d)\ge 2^{t+1}+k-t-3=\tilde{g}_2(k,d)+1$.
   \item[(3)] For integers $2\le t\le k-2$ and $d=2^{k-1}-2^t-2$ we have
   $\tilde{n}_2(k,d)\ge 2^k-2^{t+1}-1=\tilde{g}_2(k,d)+1$.
   \item[(4)] If $n_2(k,d+2)>n$ and $n-d<n_2\!\left(k-1,2\left\lceil\tfrac{d}{4}\right\rceil\right)$, then no $[n,k,d]_2 $ SO code exists.
  \end{itemize}
\end{lemma}
\begin{proof}
   Proposition~\ref{prop_SO_characterization}.(B) implies the existence of an $[\le n,k-1,d+2]_2$ subcode, which yields (1). For an $\left[n',k-1,2^t\right]_2$ code $C'$, we write $2^t=2\cdot 2^{k-2}-\sum\limits_{i=t}^{k-2} 2^i$, so that
    $$
      n'\ge g_2\!\left(k\!-\!1,2^t\right) =2\cdot\left(2^ {k-1}\!-\!1\right)-\sum_{i=t}^{k-2} \left(2^{i+1}\!-\!1\right)
      =2^k-2-2^{t+1}\cdot \left(2^{k-t-1}-1\right)\,+\,(k-t-1)=2^{t+1}+k-t-3.
    $$
    Similarly, we compute $g_2\!\left(k,2^t-2\right)=2^{t+1}+k-t-5$, so that
    $\tilde{n}_2\!\left(k,2^t-2\right)\ge n'\ge 2^{t+1}+k-t-3=g_2\!\left(k,2^t-2\right)+2=\tilde{g}_2\!\left(k,2^t-2\right)+1$, i.e.\ statement (2).
    For an $\left[n'',k-1,2^{k-1}-2^t\right]_2$ code $C''$, we write $2^{k-1}-2^t=2\cdot 2^{k-2}-2^t$, so that
    $$
       n''\ge g_2\!\left(k\!-\!1,2^{k-1}-2^t\right) =2\cdot\left(2^ {k-1}\!-\!1\right)-\left(2^{t+1}\!-\!1\right)
      =2^k-2^{t+1}-1
    $$
    and
    $$
      g_2\!\left(k,2^{k-1}-2^t-2\right) =\left(2^ k\!-\!1\right)-\left(2^{t+1}\!-\!1\right)-3=2^k-2^{t+1}-3,
    $$
    which implies statement~(3). For (4) we apply Lemma~\ref{lemma_residual_SO_code_bound} with a codeword of minimum weight $d$ (using $n_2(k,d+2)>n$ and the non-existence of codewords of weight $d+1$).
\end{proof}
Applying Lemma~\ref{lemma_non_existence_summary}.(2) with $t\in\{k-1,k-2\}$ gives \cite[Theorem 6]{li2026nonexistence} and applying Lemma~\ref{lemma_non_existence_summary}.(3) with $t=u-1$ gives \cite[Theorem 9]{li2026nonexistence}. Combining Lemma~\ref{lemma_non_existence_summary}.(1) with known lower bounds for $n_2(k,d)$ \cite{grassl_codetables} yields e.g.\ the following non-existence results.
\begin{corollary}
  \label{cor_ex_non_existence_1}
  There are no $[13,4,6]_2$, $[14,5,6]_2$, $[22,5,10]_2$, $[29,5,14]_2$, $[15,6,6]_2$, $[23,6,10]_2$, $[30,6,14]_2$, $[39,6,18]_2$, $[46,6,22]_2$, $[54,6,26]_2$, $[61,6,30]_2$, $[17,7,6]_2$, $[25,7,10]_2$, $[41,7,18]_2$, $[47,7,22]_2$, $[62,7,30]_2$, $[72,7,34]_2$, $[79,7,38]_2$, $[87,7,42]_2$, $[94,7,46]_2$, $[103,7,50]_2$, $[110,7,54]_2$, $[118,7,58]_2$, $[125,7,62]_2$,
  $[18,8,6]_2$, $[26,8,10]_2$, $[34,8,14]_2$, $[42,8,18]_2$, $[49,8,22]_2$, $[58,8,26]_2$, $[74,8,34]_2$, $[81,8,38]_2$, $[137,8,66]_2$, $[144,8,70]_2$, $[152,8,74]_2$, $[159,8,78]_2$, $[168,8,82]_2$, $[175,8,86]_2$, $[183,8,90]_2$, $[190,8,94]_2$, $[200,8,98]_2$, $[207,8,102]_2$, $[215,8,106]_2$, $[222,8,110]_2$, $[231,8,114]_2$, $[238,8,118]_2$, $[246,8,122]_2$, and $[253,8,126]_2$ SO codes.
\end{corollary}
Combining Lemma~\ref{lemma_non_existence_summary}.(4) with known lower bounds for $n_2(k,d)$ \cite{grassl_codetables} yields e.g.\ the following non-existence results.
\begin{corollary}
  \label{cor_ex_non_existence_2}
  There are no $[105,8,50]_2$, $[112,8,54]_2$, and $[129,9,62]_2$ SO codes.
\end{corollary}

For the constructive side we have:
\begin{lemma}
  \label{lemma_griesmer_code_implies_SO}
  Let $d$ be even and $k\ge 2$.
  \begin{itemize}[noitemsep]
    \item[(1)] If $d\equiv 0\pmod 4$ and a $[g_2(k,d),k,d]_2$ Griesmer code
               exists, then it is self-orthogonal.
    \item[(2)] If $d\equiv 2\pmod 4$, $d> 2^{k-1}$ and
              a $[g_2(k,d+2),k,d+2]_2$ Griesmer
               code exists, then there is a $[g_2(k,d)+1,k,d]_2$ self-orthogonal code.
    \item[(3)] If $d\equiv 2\pmod 4$, $d> 2^{k-1}$ and
              a doubly-even $[n+2,k,d+2]_2$
               code exists, then there is an $[n,k,d]_2$ self-orthogonal code.
  \end{itemize}
\end{lemma}
\begin{proof}$\,$
  \begin{itemize}
    \item[(1)] By Theorem~\ref{thm_griesmer_divisible} each such code
               is doubly-even and thus self-orthogonal.
    \item[(2)] Let $C$ be the initial Griesmer code and $\cM$ be the
               corresponding multiset of points. As observed in (1), $C$
               and $\cM$ are self-orthogonal. Since $d>2^{k-1}$ we
               have $\#\cM\ge 2^k$ and there exists a point $P$ with
               multiplicity $\cM(P)\ge 2$. With this $\cM-2\chi_P$ is SO,
               has minimum distance at least $d$, and cardinality
               $g_2(k,d+2)-2=g_2(k,d)+1$, see Subsection~\ref{subsec_Solomon_Stiffler}. Since $d>2^{k-1}\ge 2$ we have that $\cM-2\chi_P$ is spanning.
    \item[(3)] Let $C$ be the initial doubly-even code and $\cM$ be the
               corresponding multiset of points. Similar to (2) we
               conclude that $C$ is self-orthogonal and $\cM-2\chi_P$ has the desired properties for some suitable point $P$.
  \end{itemize}
\end{proof}

To treat the case $d\equiv 2\pmod 4$ for $d\le 2^{k-1}$ we first construct
special non-projective codes and then apply the idea of the Solomon--Stiffler construction, i.e.\ the removal of suitable minihypers.
\begin{lemma}
  \label{lemma_geometric_construction_1}
  For $k\ge r\ge 2$ let $V$ denote $\PG(k-1,2)$, $A\le V$ an $r$-space, and
  $B\le A$ an $(r-1)$-space. Then, $\chi_V+\chi_A-2\chi_B$ corresponds to
  a $2^{r-1}$-divisible $\left[2^k,k,2^{k-1}\right]_2$ non-projective two-weight
  code with maximum weight $2^{k-1}+2^{r-1}$.
\end{lemma}
\begin{proof}
  Consider the different possible intersections of hyperplanes with $A$ and $B$.
\end{proof}
We remark that the six non-equivalent full-length $[64,6,32]_2$ codes are all of this type (allowing $r=1$ as a degenerated case). Since the construction guarantees points of multiplicity $2$ for $r\neq k$ and self-orthogonality for $r\ge 3$, we can obtain a $[62,6,30]_2$ SO code, that turns out to be optimal. Similar to the Solomon-Stiffler construction we can remove several suitable selected subspaces. To this end, we mention that for each $k/2\le r<k$ with $k\ge 3$ there exists a partition of $\PG(k-1,q)$ into one $r$-space and $q^{r}$ $(k-r)$-spaces \cite{Beutelspacher1978,bu1980partitions}.

\begin{proposition}
  \label{prop_construction_griesmer_plus_2_singly_even}
  Let $k\ge 5$, $d\equiv 2\pmod 4$ with $d\le 2^{k-1}$, and $\varepsilon_1,\dots,\varepsilon_{k-1}$ be uniquely defined as in Lemma~\ref{lemma_parameters_griesmer_code}. Denote the largest index of the $\varepsilon_i\neq 0$ by $j$ and the second largest by $h$, where we set $h=-\infty$ if there is no second largest index.
  \begin{itemize}[noitemsep]
    \item[(1)] If $j=2$, then there exists a $[g_2(k,d)+2,k,d]_2$ SO code.
    \item[(2)] If $j=k-2$ and $d=2^{k-1}-2^{j-1}-2$, then there exists
               a $[g_2(k,d)+2,k,d]_2$ SO code.
    \item[(3)] If $3\le j\le k-3$ and $j+h\le k$, then there exists
               a $[g_2(k,d)+2,k,d]_2$ SO code.
  \end{itemize}
\end{proposition}
\begin{proof}
  We construct a corresponding multiset of points in $\PG(k-1,2)$.
  Denote the ambient space by $V$.
  \begin{itemize}
    \item[(1)] Here we have $d=2^{k-1}-2$ and a construction is given by
               Lemma~\ref{lemma_geometric_construction_1} with $r=3$.
    \item[(2)] Let $K$ be a $(k-2)$-space, $B$ a disjoint line, and $A$
               a plane that contains $B$. Let $Q$ be the point that
               equals $K\cap A$, and $P\neq Q$ be a point that is contained in $A$ but not in $B$. With this, the desired multiset of points is given by $\chi_V-\chi_K+\chi_A-2\chi_B-2\chi_P$.
    \item[(3)] If $h=2$, then choose a $j$-space $K$ and proceed as in (2).
               If $h>2$, then set $j':=\max\{j,\left\lceil k/2\right\rceil\}$, choose a $j'$-space $K'$, a $j$-space $K$ inside $K'$, and a (vector space) partition of $V$ into $K'$ and $l:=2^{j'}$ $(k-j')$-spaces $T_1,\dots,T_l$. For each $3\le i\le h$ choose an $i$-space $S_i$ inside the first $m:=h-2$ partition elements $T_r$ so that they are pairwise different. Choose a plane $A$ in $T_{m+1}$ and a
               line $B$ that is contained in $A$ as well as a point $P\notin K$ that is contained in $A$ but not in $B$. Since $h\le k-j'$ and $l\ge 2^{k/2}\ge  k/2\ge m+1$, this is always possible. With this, the desired multiset of points is given by $\chi_V-\chi_K-\sum_{i=3}^{h}\varepsilon_i \chi_{S_i}\,+\,\chi_A-2\chi_B-2\chi_P$.
  \end{itemize}
\end{proof}
We remark that the construction in Proposition~\ref{prop_construction_griesmer_plus_2_singly_even} e.g.\ yields a $[95,7,46]_2$ SO code, cf.~\cite{shi2025binary} where such a code was found by a computer search. We can construct another family of non-projective two-weight SO codes.
\begin{lemma}
  \label{lemma_non_projective_2wt_2}
  For $k\ge 4$ there exists a $\left[2^k-4,k,2^{k-1}-4\right]_2$ two-weight SO code with second weight $2^{k-1}$.
\end{lemma}
\begin{proof}
  Let $V$ denote the ambient space $\PG(k-1,2)$, $E$ a plane, and $P\in E$ a point. Then counting points for the possible hyperplane intersections yields that the multiset of points $\cM=\chi_V-\chi_E+4\cdot\chi_P$ corresponds to a linear code with the stated properties.
\end{proof}
If $C$ is an $[n,k,d]_2$ SO code with maximum weight $w$, then adding an all-one codeword of even length $n'\ge n$ (that is not already present) gives a $\left[n',k+1,\min\{d,n'-w\}\right]_2$ SO code. This easy construction is especially effective when applied to a, possibly non-projective, two-weight code. An optimal example is given by the $[124,8,60]_2$ SO code listed in Appendix~\ref{sec_generator_matrices_of_SO_codes}. See e.g.\ \cite{Calderbank1986Geometry} for projective two-weight codes and \cite{Kurz2024NonProjective} for non-projective two-weight codes. In Appendix~\ref{sec_generator_matrices_of_SO_codes} a non-projective $[183,6,92]_2$ two-weight code appears as a subcode of a $[186,8,92]_2$ SO code. Instead of adding an all-one codeword a $62$-fold repetition of a $[3,2,2]_2$ simplex code is used.

\begin{lemma}
  \label{lemma_insert}(Cf.~\cite[Remark 3.2(2)]{shi2025binary})
  Let $C$ be an $[n,k,d]_2$ SO code with a codeword of weight $w$, then we
  have $\tilde{n}_2(k,\tilde{d})\le n+\tilde{n}_2(k-1,d')$, where $\tilde{d}=\min\{d+d',w\}$ and $d'$ is arbitrary.
\end{lemma}
\begin{proof}
  Let the multiset of points $\cM$ in $\PG(k-1,2)$ correspond to $C$, $\cM'$ correspond to an $[n',k-1,d']_2$ SO code with $n'=\tilde{n}_2(k-1,d')$, and $H$ be a hyperplane with
  $\cM(H)=n-w$. Then $\widetilde{\cM}:=\cM+\cM'$, where $\cM'$ is embedded in $H$, is an $\widetilde{m}$-arc with
  $
    \widetilde{m}=\max\{n-d+n'-d',n-w+n'\}=\max\{n+n'-(d+d'),(n+n')-w\}
  $.
\end{proof}
\begin{corollary}
  \label{cor_insert}
  There exist $[47,6,22]_2$, $[55,6,26]_2$, $[80,7,38]_2$, $[82,7,40]_2$, $[88,7,42]_2$, $[90,7,44]_2$, $[147,8,72]_2$, $[155,8,76]_2$, $[171,8,84]_2$, $[176,8,86]_2$, $[178,8,88]_2$, $[201,8,98]_2$, $[203,8,100]_2$, $[208,8,102]_2$, $[210,8,104]_2$, $[216,8,106]_2$, $[223,8,110]_2$, $[232,8,114]_2$, $[239,8,118]_2$, $[247,8,122]_2$, and $[254,8,126]_2$ SO codes.
\end{corollary}
\begin{proof}
  Apply Lemma~\ref{lemma_insert} with $\mathcal{RM}_2(1,k-1)$ as $C$ and $w=2^{k-1}$.
\end{proof}
Lemma~\ref{lemma_insert} also yields $\left[\tilde{n}_2(7,d),7,d\right]_2$ SO codes for all even $48\le d\le 64$ with $d\neq 52$. The question arises whether there exist other SO codes, with
length near the Griesmer bound and a large maximum weight, besides first-order Reed--Muller codes. The generator matrices
$$
  \left(\begin{smallmatrix}
111111000000000111111110000000011111111111111110000000001100101101000000000\\
000000111000000000000001111111111111111111111110001110000010011100100000100\\
000000000111111111111111111111100000000111111110000000000001110100010000010\\
000111000001111000011110000111100001111000011111111110000101101010001000001\\
011001011010011001100110011001100110011001100110110011111010001110000100000\\
100111111101101010100011101010101010111010001011000110111000100110000010111\\
101010001000110101001010110010110100101010110100011001011101110010000001111\\
  \end{smallmatrix}\right)
$$
and
$$
  \left(\begin{smallmatrix}
111100000000111100001111111111100000000011110000111111110010010111000000000\\
000011110000000011111111111100011100000000001111111111110001001110100000000\\
000000001111111111110000111100000011100011111111000011110000111010010000000\\
111111111111111111111111111100000000011100000000000000000011111110001000100\\
001100110011001100110011001101101111101100110011001100111111100100000100011\\
011101000111110100010100111110110101110101010101010101011001001100000010111\\
100110011011011001100101010111011010001010011001000101101101001010000001101\\
  \end{smallmatrix}\right)
$$
yield optimal self-orthogonal $[75,7,36]_2$ codes with weight distributions $36^{75} 40^{51}   60^1$ and $36^{76} 40^{50} 64^1$. The orders of their automorphism groups are given by $1440$ and $96$, respectively. However, using them as code $C$ in Lemma~\ref{lemma_insert} does not yield any optimal codes.
\todo{Study possible geometric explanations. Do they come from e.g.\ $8$-divisible $[n,k,40]_2$ codes? Comment on the (unique) large maximum weight and how it might be eventually used in the construction in Lemma~\ref{lemma_insert}, while it does not yield any very good codes in this special situation. They are also codes that can serve as examples for SO codes found by prescribing automorphism groups. The example in \texttt{BKLC}, referred to in \cite{shi2025binary}, has a trivial? or small automorphism group.}
\begin{lemma}
  \label{lemma_no_10_5_4_SO}
  No $[10,5,4]_2$ SO code exists.
\end{lemma}
\begin{proof}
  Let $\left(I_5|G'\right)$ be a generator matrix of such a code $C$, where $I_5$ denotes the $5\times 5$ unit matrix.
  Since $C$ is even and has minimum distance $4$ the weights of the rows of $G'$ are contained in $\{3,5\}$. Orthogonality rules out weight $5$ and
  implies that each pair of different rows has exactly two $1$s in common. There cannot be three rows of $G'$ with the same two $1$s in common, since choosing a fourth row would be impossible. However, for one row there are just $3$ ways to choose two $1$s, so that there are at most four rows.
\end{proof}
Note that there are numerous ways to show this result, e.g.\ using results for self-dual codes. We remark that $[10,5,4]_2$ codes with weight distributions
$4^{10}   5^{16}   8^5$ and $4^{15}   6^{15}   10^1$ do exist. However, the only doubly-even $[\le 10,4]_2$ codes have weight distributions $4^{14}   8^1$ and $4^{10}   8^5$, cf.~\cite{KiermaierKurz2023}.

\begin{lemma}
  We have $\tilde{n}_2(k,2)=2k$ for all $k\ge 1$.
\end{lemma}
\begin{proof}
  The lower bound is given by $n\ge 2k$ and a construction is given by doubling the multiset of points given by a unit matrix.
\end{proof}

\begin{proposition}
  For even $d\ge 2$ we have
  $$
    \tilde{n}_2(3,d)=\left\{\begin{array}{rcl}
    \tilde{g}_2(3,d)+1 &:& d=2,\\
    \tilde{g}_2(3,d) &:& \text{otherwise}.
    \end{array}\right.
  $$
\end{proposition}
\begin{proof}
  Assume $d\ge 4$, so that all lower bounds are given by the Griesmer SO bound.
  For doubly-even $d\ge 4$ and for singly-even $d\ge 6$ there exist Griesmer codes with minimum distance $d$ or $d+2$, so that we can apply Lemma~\ref{lemma_griesmer_code_implies_SO}.
\end{proof}
Note that it suffices to provide constructions for $d\le 2+2^2=6$.
For even $d\ge 2$ we have $n_2(3,d)=g_2(3,d)$.

\begin{proposition}
  For even $d\ge 2$ we have
  $$
    \tilde{n}_2(4,d)=\left\{\begin{array}{rcl}
    \tilde{g}_2(4,d)+2 &:& d=2,\\
    \tilde{g}_2(4,d)+1 &:& d=6,\\
    \tilde{g}_2(4,d) &:& \text{otherwise}.
    \end{array}\right.
  $$
\end{proposition}
\begin{proof}
  Assume $d\ge 4$.
  Besides the Griesmer SO bound, the lower bounds are given by Corollary~\ref{cor_ex_non_existence_1} for $d=6$.
  Doubling a $[7,4,3]_2$ code yields a construction for $d=6$.
  For doubly-even $d\ge 4$ and for singly-even $d\ge 10$ there exist Griesmer codes, so that we can apply Lemma~\ref{lemma_griesmer_code_implies_SO}.
\end{proof}
Note that it suffices to provide constructions for $d\le 6+2^3=14$.
For even $d\ge 2$ we have $n_2(4,d)=g_2(4,d)$.

\begin{proposition}
  For even $d\ge 2$ we have
  $$
    \tilde{n}_2(5,d)=\left\{\begin{array}{rcl}
    \tilde{g}_2(5,d)+3 &:& d=2,\\
    \tilde{g}_2(5,d)+2 &:& d=4,\\
    \tilde{g}_2(5,d)+1 &:& d\in\{6,10,14\},\\
    \tilde{g}_2(5,d) &:& \text{otherwise}.
    \end{array}\right.
  $$
\end{proposition}
\begin{proof}
  Assume $d\ge 4$.
  Besides the Griesmer SO bound, the lower bounds are given by Lemma~\ref{lemma_no_10_5_4_SO} for $d=4$ and Corollary~\ref{cor_ex_non_existence_1} for $d\in\{6,10,14\}$.
  Constructions for $d\in\{4,6,10,14\}$ can be found in \cite{BouyuklievOstergaard2006Binary}. For doubly-even $d\ge 8$ and for singly-even $d\ge 18$ there exist Griesmer
  codes, so that we can apply Lemma~\ref{lemma_griesmer_code_implies_SO}.
\end{proof}
Note that it suffices to provide constructions for $d\le 14+2^4=30$.
For even $d\ge 2$ we have $n_2(5,d)=g_2(5,d)+1$ for $d\in\{4,6\}$ and $n_2(5,d)=g_2(5,d)$ otherwise.

\begin{lemma}
  \label{lemma_no_31_6_14_SO}
  No $[31,6,14]_2$ SO code exists.
\end{lemma}
\begin{proof}
  Assume that $C$ is a $[31,6,14]_2$ SO code and $\cM$ its corresponding multiset of points. From Lemma~\ref{lemma_non_existence_summary}.(2) we conclude that $C$ is of full length. Since projection through a point for $\cM$ corresponds to taking a subcode in $C$ and $\tilde{n}_2(5,14)=30$, we have $\cM(P)\in\{0,1\}$ for all points $P$. The doubly-even weight subcode of $C$ is a $[31,5,16]_2$ code, i.e.\ a $5$-dimensional simplex code. Thus, the non-zero doubly-even codewords have weight $16$. Using the residual SO code argument in Lemma~\ref{lemma_residual_SO_code_bound} and Lemma~\ref{lemma_residual_SO_large_weight} we conclude that the only possible singly-even weight is $14$. Using the MacWilliams equations we
  obtain a contradiction.
\end{proof}
There exists an even $[30,6,14]_2$ code with weight distribution $14^{31} 16^{31} 30^1$.

\begin{proposition}
  For even $d\ge 2$ we have
  $$
    \tilde{n}_2(6,d)=\left\{\begin{array}{rcl}
    \tilde{g}_2(6,d)+4 &:& d=2,\\
    \tilde{g}_2(6,d)+2 &:& d\in\{4,14\},\\
    \tilde{g}_2(6,d)+1 &:& d\in\{6,8,10,12,18,20,22,26,30\},\\
    \tilde{g}_2(6,d) &:& \text{otherwise}.
    \end{array}\right.
  $$
\end{proposition}
\begin{proof}
  Assume $d\ge 4$.
  Besides the Griesmer SO bound, the lower bounds are given by
  $n\ge 2k$ for $d=4$, $\tilde{n}_2(6,d)\ge n_2(6,d)$ for $d\in\{8,12,20\}$, Lemma~\ref{lemma_no_31_6_14_SO} for $d=14$, and Corollary~\ref{cor_ex_non_existence_1} for $d\in\{6,10,18,22,26,30\}$.
  Constructions for $d\le 22$ can be found in \cite{BouyuklievOstergaard2006Binary} and \cite[Appendix B]{Bouyukliev2006}.\footnote{For the unique self-orthogonal $[11,5,4]_2$ code \cite{BouyuklievOstergaard2006Binary} states its weight distribution $4^{10}   6^{16}   8^5$, its order $1920$ of the automorphism group as well as its type, and refers to \cite{pless1972classification}, where it is mentioned as $\overline{B}_{11}$ occurring as a truncation of the unique self-orthogonal $[12,6,4]_2$ code $B_{12}$. While both codes are projective, the parameters of the first one are omitted in the list of \cite[Appendix B]{Bouyukliev2006}. We remark that there are exactly three non-equivalent $[11,5,4]_2$ codes with weight distribution $4^{10}  6^{16} 8^5$ having different orders of the automorphism group and that the doubly-even subcode of the $[11,5,4]_2$ code is a $4$-divisible $[10,4,4]_2$ code with weight distribution $4^{10} 8^5$.} For $d\in\{26,30\}$ Proposition~\ref{prop_construction_griesmer_plus_2_singly_even} gives examples. For doubly-even $d\ge 24$ and for singly-even $d\ge 34$ there exist Griesmer codes, so that we can apply Lemma~\ref{lemma_griesmer_code_implies_SO}.
\end{proof}

Note that it suffices to provide constructions for $d\le 30+2^5=62$.
For even $d\ge 2$ we have $n_2(6,d)=g_2(6,d)+1$ for $d\in\{4,\dots,14\}\cup\{20\}$ and $n_2(6,d)=g_2(6,d)$ otherwise.

Similar to Lemma~\ref{lemma_no_31_6_14_SO} we obtain the following non-existence results.
\begin{lemma}
  \label{lemma_no_SO_code_via_weight_restrictions}
  No $[32,7,14]_2$, $[56,7,26]_2$, and $[63,7,30]_2$ SO codes exist.
\end{lemma}
\begin{proof}
  Assume that $C$ is an $[n,k,d]_2$ SO code with parameters from the above list and $\cM$ its corresponding multiset of points. From Lemma~\ref{lemma_non_existence_summary}.(1) we conclude that $C$ is of full length.
  Projection through points and the known values for $\tilde{n}_2(6,d)$ yield $\gamma(\cM)\le 1$, i.e., $C$ is projective. With this, Lemma~\ref{lemma_residual_SO_large_weight} excludes weights $n-1$ and $n-2$. Denote the doubly-even weight subcode of $C$ by $C_4$. For our parameters it turns out that it is a doubly-even $[n,k-1,d+2]_2$ Griesmer code with $d+2\le 2^{k-2}$, i.e., it has to be of Solomon--Stiffler type $\left[1;\varepsilon_1,\dots,\varepsilon_{k-2}\right]$, so that we can compute the possible weight distributions of $C_4$ from the $\varepsilon_i$. Since
  $C_4$ contains all doubly-even codewords from $C$, this restricts the possible
  weights of $C$ as does the residual SO code argument in Lemma~\ref{lemma_residual_SO_code_bound}.
  \begin{itemize}
    \item For $[32,7,14]_2$ the residual SO code argument in Lemma~\ref{lemma_residual_SO_code_bound} excludes weights $\{18, 22, 26\}$ and $C_4$ has weight distribution $16^{62} 32^1$, so that the non-zero weights of $C$ have to be contained in $\{14,16,32\}$. With this, the MacWilliams equations give a contradiction for the existence of $C$.\footnote{Alternatively, $\tilde{n}_2(6,14)=32$ forces $\cM(P)=0$ under projection, which is impossible.}
    \item For $[56,7,26]_2$ the residual SO code argument excludes $\{34,42,46,50\}$
    and $C_4$ has weight distribution $28^{56} 32^{7}$, so that the non-zero weights of $C$ have to be contained in $\{26,28,30,32,38\}$. With this, the MacWilliams equations give a contradiction for the existence of $C$.\footnote{More precisely, using $A_{28}=56$, $A_{32}=7$, and $B_2=0$ yields $A_{26} = \tfrac{136}{3}\notin\N$. Alternatively, one can deduce the existence of a codeword of weight $26$ from the Griesmer bound and exploit one of the facts that the residual $[30,6,14]_2$ code contains a codeword of weight $30$ and does not contain a codeword of weight $20$. However, both statements would need a small separate proof.}
    \item For $[63,7,30]_2$ the residual SO code argument excludes $\{34,38,42,46,50,54,58\}$
    and $C_4$ has weight distribution $32^{63}$, so that the non-zero weights of $C$ have to be contained in $\{30,32\}$. With this, the MacWilliams equations give a contradiction for the existence of $C$.
  \end{itemize}
\end{proof}
We remark that the unique even $[30,6,14]_2$ code has weight distribution $14^{31} 16^{31} 30^1$. The non-existence of a $[62,7,30]_2$ SO code was also shown in \cite[Proposition 6.4]{shi2025binary} and the non-existence of a $[63,7,30]_2$ SO code in \cite[Remark 6.14]{shi2025binary} by an exhaustive computer search.

\begin{lemma}
  \label{lemma_no_33_7_14_SO}
  No $[33,7,14]_2$ SO code exists.
\end{lemma}
\begin{proof}
  Let $C$ be a $[33,7,14]_2$ SO code and $\cM$ be its corresponding multiset of points. From Lemma~\ref{lemma_no_SO_code_via_weight_restrictions} we conclude that $C$ has full length. The subcode $C_4$ of doubly-even weight codewords is either a $[32,6,16]_2$ or
  a $[33,6,16]_2$ code. Since the residual codes of $C$ have to be even, there is no codeword of weight $32$, which excludes the $[32,6,16]_2$ Griesmer code with
  weight distribution $16^{62} 32^1$ as candidate for $C_4$.
  Due to the residual code argument in Lemma~\ref{lemma_residual_code_bound} the $[33,6,16]_2$ code $C_4$ cannot contain a codeword with weight in $\{20,24,28\}$. However, $C_4$ cannot consist of non-zero codewords of weight $16$ only.
\end{proof}
We remark that there exist three non-equivalent full-length $[33,6,16]_2$ codes -- none of them is doubly-even.

\begin{proposition}
  For even $d\ge 2$ we have
  $$
    \tilde{n}_2(7,d)=\left\{\begin{array}{rcl}
    \tilde{g}_2(7,d)+5 &:& d=2,\\
    \tilde{g}_2(7,d)+3 &:& d\in\{4,14\},\\
    \tilde{g}_2(7,d)+2 &:& d\in\{6,10,16,18,26,28,30\},\\
    \tilde{g}_2(7,d)+1 &:& d\in\{8,12,20,22,24,34,\dots,46,50,54,58,62\},\\
    \tilde{g}_2(7,d) &:& \text{otherwise}.
    \end{array}\right.
  $$
\end{proposition}
\begin{proof}
  Assume $d\ge 4$.
  Besides the Griesmer SO bound, the lower bounds are given by
  $n\ge 2k$ for $d=4$, $\tilde{n}_2(7,d)\ge n_2(7,d)$ for $d\in\{8,12,16,20,24,28,36,40,44\}$,
  Lemma~\ref{lemma_no_33_7_14_SO} for $d=14$,
  Lemma~\ref{lemma_no_SO_code_via_weight_restrictions} for $d\in\{26,30\}$,
  and Corollary~\ref{cor_ex_non_existence_1} for $d\in\{6,10,18,22,34,38,42,46,50,54,58,62\}$.
  Constructions for $d\le 16$ can be found in \cite{BouyuklievOstergaard2006Binary} and \cite{pless1972classification}.
  For constructions with $18\le d\le 36$ 
  we refer to \cite{shi2025binary}.
  For $38\le d\le 44$ Corollary~\ref{cor_insert} gives examples.
  For singly-even $46\le d\le 62$ Proposition~\ref{prop_construction_griesmer_plus_2_singly_even} gives examples.
  For doubly-even $d\ge 48$ and for singly-even $d\ge 66$ there exist Griesmer
  codes, so that we can apply Lemma~\ref{lemma_griesmer_code_implies_SO}.
\end{proof}
\todo{Mention the example of the explicit generator matrices for $[75,7,36]_2$ SO codes. Try to explain more of the examples for $18\le d\le 36$ from \cite{shi2025binary}.}
Note that it suffices to provide constructions for $d\le 62+2^6=126$.
For even $d\ge 2$ we have $n_2(7,d)=g_2(7,d)+2$ for $d\in\{14,16,26,28\}$, $n_2(7,d)=g_2(7,d)+1$ for $d\in\{4,\dots,12\}\cup\{18,\dots,24\}\cup\{30\}\cup\{36,\dots,44\}$, and $n_2(7,d)=g_2(7,d)$ otherwise.

\section{Optimal binary self-orthogonal codes with dimension $\mathbf{8}$}
\label{sec_opt_8}

In this section we completely determine $\tilde{n}_2(8,\cdot)$, see Theorem~\ref{main_thm}.
The existence of a linear code with a prescribed minimum distance and a prescribed subgroup of the automorphism group can be formulated as a system of Diophantine linear equations, see e.g.\ \cite{braun2005optimal}. Here we use the fact that each doubly-even binary linear code is self-orthogonal and solve the corresponding systems with integer linear programming.
\begin{lemma}
  \label{lemma_construction_aut}
  There exist $[82,8,38]_2$, $[84,8,40]_2$, $[113,8,54]_2$, $[115,8,56]_2$, $[138,8,66]_2$, $[140,8,68]_2$, $[145,8,70]_2$, $[147,8,72]_2$, $[153,8,74]_2$, $[155,8,76]_2$, $[160,8,78]_2$, $[162,8,80]_2$, $[169,8,82]_2$,
  $[171,8,84]_2$, $[176,8,86]_2$, $[178,8,88]_2$, $[184,8,90]_2$, $[186,8,92]_2$, $[68,9,30]_2$, $[70,9,32]_2$, $[84,9,38]_2$, $[93,9,42]_2$, $[100,9,46]_2$, $[116,9,54]_2$, $[118,9,56]_2$, $[124,9,58]_2$, and $[131,9,62]_2$ SO codes.
\end{lemma}
\begin{proof}
  In Appendix~\ref{sec_generator_matrices_of_SO_codes} we have listed generator matrices of non-projective doubly-even $[84,8,40]_2$, $[115,8,56]_2$, $[140,8,68]_2$, $[147,8,72]_2$, $[155,8,76]_2$, $[171,8,84]_2$, $[186,8,92]_2$, $[70,9,32]_2$, and $[118,9,56]_2$ codes. For the parameters $[162,8,80]_2$, $[178,8,88]_2$, $[86,9,40]_2$, $[95,9,44]_2$, $[102,9,48]_2$, $[126,9,60]_2$, and $[133,9,64]_2$
  it turns out that the codes stored in the \texttt{BKLC} database \cite{MagmaBKLC} are doubly-even and non-projective.
  These $[n,k,d]_2$ codes are SO themselves and decreasing the point multiplicity of a suitable point by $2$ gives an $[n-2,k,d-2]_2$ SO code.
\end{proof}
For singly-even self-orthogonal multisets of points we can fix a point $P$ and require $\#\cM-\cM(H)\equiv 0\pmod 4$ for every hyperplane $H\supseteq P$ and
$\#\cM-\cM(H)\equiv 2\pmod 4$ otherwise, see Proposition~\ref{prop_SO_characterization}.(B). Alternatively, we can start with suitable candidates for the subcode of the doubly-even codewords  and use \texttt{LinCode} \cite{bouyukliev2021computer} to search for an extension to the full SO code. We present examples of $[82,8,38]_2$ and $[113,8,54]_2$ SO codes obtained this way in Appendix~\ref{sec_generator_matrices_of_SO_codes}. Next we provide the necessary non-existence results for $[n,8,d]_2$ SO codes.

\begin{lemma}
  \label{lemma_no_65_8_30_SO}
  No $[65,8,30]_2$ SO code exists.
\end{lemma}
\begin{proof}
  Let $C$ be a $[65,8,30]_2$ SO code. From $n_2(8,30)=65$ we conclude that $C$ is of
  full length, so that Lemma~\ref{lemma_residual_SO_large_weight} yields that there are no codewords of weight $64$. Denote the doubly-even weight subcode of $C$ by $C_4$. Since a $[64,7,32]_2$ code is a Griesmer code with weight distribution
  $32^{126} 64^1$, $C_4$ is a full-length $4$-divisible $[65,7,32]_2$ code. Applying
  the residual code argument from Lemma~\ref{lemma_residual_code_bound} to $C_4$
  yields that there are no codeword weights strictly between $32$ and $64$. However, since there is no codeword of weight $64$ in $C\supset C_4$, the MacWilliams equations contradict the existence of $C_4$.\footnote{Alternatively, one can use the  fact that one-weight codes have to be repeated simplex codes and that the parameters do not match.}
\end{proof}
The non-existence of $[65,8,30]_2$ SO codes was previously shown via exhaustive enumeration using \texttt{Q-EXTENSION} \cite{bouyukliev2007qextension}, see \cite[Remark 20]{li2026nonexistence}.

\begin{lemma}
  \label{lemma_no_SO_code_via_weight_restrictions2}
  No $[96,8,46]_2$, $[120,8,58]_2$, and $[127,8,62]_2$ SO codes exist.
\end{lemma}
\begin{proof}
  Assume that $C$ is an $[n,k,d]_2$ SO code with parameters from the above list and $\cM$ its corresponding multiset of points. From Lemma~\ref{lemma_non_existence_summary}.(1) we conclude that $C$ is of full length.
  Projection through points and the known values for $\tilde{n}_2(7,d)$ yield $\gamma(\cM)\le 1$, i.e., $C$ is projective. With this, Lemma~\ref{lemma_residual_SO_large_weight} excludes weights $n-1$ and $n-2$. Denote the doubly-even weight subcode of $C$ by $C_4$. For our parameters it turns out that it is a doubly-even $[n,k-1,d+2]_2$ Griesmer code with $d+2\le 2^{k-2}$, i.e., it has to be of Solomon--Stiffler type $\left[1;\varepsilon_1,\dots,\varepsilon_{k-2}\right]$, so that we can compute the possible weight distributions of $C_4$ from the $\varepsilon_i$. Since
  $C_4$ contains all doubly-even codewords from $C$, this restricts the possible
  weights of $C$ as does the residual SO code argument in Lemma~\ref{lemma_residual_SO_code_bound}.
  \begin{itemize}
    \item For $[96,8,46]_2$ the residual SO code argument excludes all singly-even weights in $\{50,\dots,90\}$ and $C_4$ has weight distribution $48^{124} 64^{3}$, so that the non-zero weights of $C$ have to be contained in $\{46,48,64\}$. With this, the MacWilliams equations give a contradiction for the existence of $C$.
    \item For $[120,8,58]_2$ the residual SO code argument in Lemma~\ref{lemma_residual_SO_code_bound} and Lemma~\ref{lemma_residual_SO_large_weight} excludes all singly-even weights in $\{62,66\}\cup\{74,\dots,118\}$ and $C_4$ has weight distribution $60^{112} 64^{15}$, so that the non-zero weights of $C$ have to be contained in $\{58,60,64,70\}$. Assume that $C'$ is the full-length projective even $[50,7,24]_2$ residual code of a codeword $c'$ of weight $70$. From the residual code argument in Lemma~\ref{lemma_residual_code_bound} we conclude that the non-zero weights of $C'$ are contained in $\{24,32,48,50\}$. Since weight $24$ has to occur (due to the Griesmer bound) and weight $26$ does not occur, there also cannot be a codeword of weight $50$. With this, the MacWilliams equations for
    $C'$ yield a contradiction ($A_{48}<0$). Thus, $C$ cannot contain a codeword of
    weight $70$. With this, the MacWilliams equations give a contradiction for the existence of $C$.\footnote{Alternatively, we may use the MacWilliams equations with $B_1=B_2=0$, $A_{60}=112$, $A_{64}=15$, and $\{58,70\}$ as possibilities for the singly-even weights to conclude $A_{70}\notin\N$.}
    \item For $[127,8,62]_2$ the residual SO code argument excludes all singly-even weights $66\le w\le 122$ and $C_4$ has weight distribution $64^{127}$, so that the non-zero weights of $C$ have to be contained in $\{62,64\}$. With this, the MacWilliams equations give a contradiction for the existence of $C$.
  \end{itemize}
\end{proof}
The non-existence of $[96,8,46]_2$ SO codes was previously shown via exhaustive enumeration using \texttt{Q-EXTENSION} \cite{bouyukliev2007qextension}, see \cite[Remark 20]{li2026nonexistence}. There is a unique $[50,7,24]_2$ code \cite{kurz2020classification}. However,  its weight distribution is given by $24^{108} 32^{19}$ and it is non-projective.

\begin{lemma}
  \label{lemma_SO_non_existence_LinCode}
  There exist no $[90,8,42]_2$, $[97,8,46]_2$, and $[52,9,22]_2$ SO codes.
\end{lemma}
\begin{proof}
  Using \texttt{LinCode} \cite{bouyukliev2021computer} we have classified all
  candidates for the subcode of the doubly-even codewords and checked that they
  cannot be extended to the full SO code. The parameters and the number of non-equivalent codes are stated in Table~\ref{table_subcode_enumeration}.
\end{proof}

\begin{table}[htp!]
  \begin{center}
    \begin{tabular}{rrrrrr}
      \hline
      $[n,k-1,d+2]_2$          & $[90,7,44]_2$ & $[96,7,48]_2$ & $[97,7,48]_2$ & $[120,7,60]_2$ & $[121,7,60]_2$ \\
      \#\text{ codes}          & 6 & 1 &   & 1 &   \\
      \#\text{ $4$-div. codes} & 3 & 1 & 1 & 1 & 6 \\
      \hline
      $[n,k-1,d+2]_2$          & $[51,8,24]_2$ & $[52,8,24]_2$ \\
      \#\text{ codes}          & 1 & 8 \\
      \#\text{ $4$-div. codes} & 1 &2 \\
      \hline
    \end{tabular}
    \caption{Enumeration of (almost) optimal full-length codes.}
    \label{table_subcode_enumeration}
  \end{center}
\end{table}

The non-existence of $[52,9,22]_2$ SO codes was also shown in \cite{shimizu2026selforthogonal}. \texttt{LinCode} needs six hours for $[90,8,42]_2$ and five minutes for
$[97,8,46]_2$. Checking possible extensions of the six doubly-even full-length $[121,7,60]_2$ codes using \texttt{LinCode} requires 88~hours of computation time. For the unique $[120,7,60]_2$ code we have aborted the computation
after 12~days and a memory usage of more than 30~GB. So we pursue a different approach based on the split weight enumerator.\footnote{Alternatively, one may use any linear constraint solver that exploits the $13872660480$ automorphisms of the $[120,7,60]_2$ code.}

\begin{lemma}
  \label{lemma_no_121_8_58_SO}
  There exists no $[121,8,58]_2$ SO code.
\end{lemma}
\begin{proof}
  Let $C$ be a $[121,8,58]_2$ SO code and let $\cM$ be its corresponding multiset of points. From Lemma~\ref{lemma_no_SO_code_via_weight_restrictions2} we conclude that $C$ has full length. Since $\tilde{n}_2(7,58)=119$, we have $\gamma(\cM)\le 2$.
  If $c\in C$ is a codeword of weight $118$, then the residual
  code is a full-length $[3,k',2]_2$ code, so that $k'=2$. The subcode $C'$ of $C$ shortened onto $\supp(c)$ has dimension $6$ and its non-zero codewords have weights contained in $\{58,60,118\}$. Note that $C'$ is self-orthogonal itself and the subcode spanned by the doubly-even codewords has dimension at least $5$, which e.g.\ contradicts the classification of one-weight codes in \cite{Bonisoli1984}. Thus, there is no codeword of weight $118$.
  From the residual code argument in Lemma~\ref{lemma_residual_SO_code_bound} we conclude that the weights of the non-zero singly-even codewords are contained in
  $$
    \{\underline{58},62,70,74,78,86,94,102\}.
  $$
  Since no $[121,8,60]_2$ code exists weight $58$ has to occur.
  By $C_4$ we denote the subcode consisting of the doubly-even codewords of $C$ and by $\cM_4$ the corresponding multiset of points.
  Using \texttt{LinCode} we have exhaustively enumerated all doubly-even $[\le 121,7,60]_2$ codes. Their weight distributions and some further data are given as follows:
  \begin{center}
    \begin{tabular}{llrr}
    \hline
    $n$ & w.e. & $\#\operatorname{Aut}$ & $B_2$ \\
    \hline
    $120$\footnote{A $[120,7,60]_2$ Griesmer code has to be of Solomon--Stiffler type $[1;0,0,1,0,0,0]$.}
          & $60^{112} 64^{15}$           & $13872660480$ &  $0$ \\
    \hline
    $121$ & $60^{108} 64^{15} 76^4$      & $483840$      & $15$ \\
          & $60^{104} 64^{15} 68^8$      & $258048$      &  $7$ \\
          & $60^{96}  64^{31}$           & $884736$      &  $3$ \\
          & $60^{98}  64^{27} 68^2$      & $258048$      &  $4$ \\
          & $60^{105} 64^{14} 68^7 72^1$ & $225792$      &  $8$ \\
          & $60^{100} 64^{23} 68^4$      & $737280$      &  $5$ \\
    \hline
    \end{tabular}
  \end{center}
  Thus, the weights of the non-zero doubly-even codewords in $C$ are contained in
  $$
    \{\underline{60},64,68,72,76\}.
  $$
  Since no $[121,7,64]_2$ code exists weight $60$ has to occur.
  Using \texttt{LinCode} we have exhaustively enumerated the following candidates for the residual codes $\Res(C;c)$:
  \begin{center}
    \begin{tabular}{lrrrrrrr}
    \hline
    $\wt(c)$            & 58 & 62 & 70 & 74 & 78 & 86 & 94 \\
    $\Res(C;c)$         & $[63,7,30]_2$ & $[59,7,28]_2$ & $[51,7,24]_2$ & $[47,7,22]_2$ &
                          $[43,7,20]_2$ & $[35,7,16]_2$ & $[27,7,12]_2$\\
    $\# \text{ (even)}$ & 20 & 143 & 68 & 1 & 7 & 4 & 1 \\
    \hline
    \end{tabular}
  \end{center}
  Having the possible weight distributions of these residual codes at hand we iteratively apply the linear programming method for the split weight enumerator w.r.t.\ the partition given by the support of a codeword $c$ and its complement. Each time when we say that we apply the split weight enumerator for $\wt(c)$ we consider all pairs of weight distributions for $C_4$, $\Res(C;c)$ and prescribe them. We then state the conclusions that we draw from the respective set of infeasible linear programs.

  The split weight enumerator for $\wt(c)=58$ excludes the weights $\{72,102\}$ for $C$ and the cases $A_{60}\in\{100,104,105\}$ for $C$, $C_4$. Applying the split weight enumerator for weights $94$, $86$, $78$,
  $62$ excludes these weights. Applying the split weight enumerator for $\wt(c)=70$ shows that there cannot be such a codeword.
  Applying the split weight enumerator for $\wt(c)=74$ yields $A_{60}=96$, which implies the unique weight distribution $58^{108} 60^{96} 64^{31} 74^{20}$ with $B_2=5$. 
  Since $C_4\subseteq C$, we have $C^\perp\subseteq C_4^\perp$, so $5=B_2(C) \le  B_2\!\left(C_4\right)$, which contradicts $B_2(C_4)=3$ for weight distribution $60^{96} 64^{31}$. Thus, there is no codeword of weight $74$. Applying the split weight enumerator for $\wt(c)=58$ again, yields a contradiction.
\end{proof}\todo{Give a precise general statement for the relation between the $b_2$-values of codes and its subcodes.}
Note that we have not used the computer-aided classified residual and subcodes explicitly, but only used lists of possible weight distributions. It seems very likely that one can iteratively apply the linear programming method for the split weight enumeration with partitions of reasonable size to obtain a \textit{computer-free} proof. This should also be the case for the two parameter sets in Lemma~\ref{lemma_SO_non_existence_LinCode} with $k=8$.\footnote{Proofs that do not use exhaustive computer enumerations of codes were found with AI tools e.g.\ for the cases $[90,8,42]_2$, $[97,8,46]_2$, and $[121,8,58]_2$. The proofs have been checked by independent AI reviews and exact computer checks, but not yet by a human expert, see the GitHub repository \url{https://github.com/yygotm/self-orthogonal-dim8} by Hirotaka Shimizu.\label{fn_AI}}

We remark that there exist exactly six non-equivalent $[119,7,58]_2$ SO codes with maximum weight at most $76$. Five of them violate $A_{66}=0$ or $A_{62}\le 20$. The generator matrix
$$
  \left(\begin{smallmatrix}
11000011001111000111100000000111100001111111100001111111100000000000000001111111100000000111111111111111100001011100000\\
00110000111111000000011110000000011111111111100000000000011111111000000000000000011111111111111111111111100000111010000\\
00001111110011000000000001111111111110000111100000000000000000000111111111111111111111111000000001111111100001101001000\\
00000000000000100000000000000000000000000000011001111111111111111111111111111111111111111111111111111111111110000000100\\
00000000000000010111111111111111111111111111100110000111100001111000011110000111100001111000011110000111111110000000010\\
11111111111111001001100110011001100110011001111110011001100110011001100110011001100110011001100110011001100110000000001\\
11010101010100101010001010101010111010101011101010101010111010001000111010100011101010101011101000101010101011010000000\\
  \end{smallmatrix}\right)
$$
yields a $[119,7,58]_2$ SO code with weight distribution $58^{52} 60^{56} 62^8 64^7   74^4$ and an automorphism group of order $32256$. Interestingly enough, it contains $14$ duplicated columns. In Table~\ref{table_SO_enumeration} we state the number of $[n,k,58]_2$ SO codes with non-zero weights in $\{58,60,62,64,68,70,72,74,76,78,$ $84,86,94\}$ for moderate parameters. The \texttt{LinCode} computation took 22~hours for the stated values with $n\le 118$.

\begin{table}[htp]
  \begin{center}
    \begin{tabular}{rrrrrrrrrrrrrrr}
    \hline
    \textbf{k/n} & \textbf{105} & \textbf{106} & \textbf{107} & \textbf{108} & \textbf{109} & \textbf{110} & \textbf{111} & \textbf{112} & \textbf{113}& \textbf{114}& \textbf{115}& \textbf{116}& \textbf{117}& \textbf{118}\\
    \textbf{3} &5&10&18&29&49&73&110&157&220&294&393&495\\
    \textbf{4} &&&&&&0&2&10&47&195&701&2322&6943\\
    \textbf{5} &&&&&&&&&&0&2&38&754&17397\\
    \textbf{6} &&&&&&&&&&&&0&1&26\\
    \textbf{7} &&&&&&&&&&&&&&0\\
    \hline
    \end{tabular}
    \caption{Number of non-equivalent full-length $[n,k,58]_2$ SO codes with non-zero weights in $\{58,60,62,64,68,70,72,74,76,78,84,86,94\}$.}
    \label{table_SO_enumeration}
  \end{center}
\end{table}
\todo{Can more values for Table~\ref{table_SO_enumeration} be computed? Might be really hard.}

\begin{theorem}
  \label{main_thm}
  For even $d\ge 2$ we have
  $$
    \tilde{n}_2(8,d)=\left\{\begin{array}{rcl}
    \tilde{g}_2(8,d)+6 &:& d=2,\\
    \tilde{g}_2(8,d)+4 &:& d=4,\\
    \tilde{g}_2(8,d)+3 &:& d\in\{14,26,28,30,32,42,46,58,60\},\\
    \tilde{g}_2(8,d)+2 &:& d\in\{6,10,12,16,\dots,22,34,\dots,40,44,48,\dots,56,62\},\\
    \tilde{g}_2(8,d)+1 &:& d\in\{8,24,66,\dots,94,98,\dots,106,110,114,118,122,126\},\\
    \tilde{g}_2(8,d) &:& \text{otherwise}.
    \end{array}\right.
  $$
\end{theorem}
\begin{proof}
  Assume $d\ge 4$. Besides the Griesmer SO bound, the lower bounds are given by
  $n\ge 2k$ for $d=4$, $\tilde{n}_2(8,d)\ge n_2(8,d)$ for doubly-even $8\le d\le 60$, $68\le d\le 92$,  and $d\in\{100,104\}$,
  Lemma~\ref{lemma_no_65_8_30_SO} for $d=30$,
  Lemma~\ref{lemma_SO_non_existence_LinCode} for $d\in\{42,46\}$,
  Corollary~\ref{cor_ex_non_existence_2} for $d\in\{50,54\}$,
  Lemma~\ref{lemma_no_121_8_58_SO} for $d=58$,
  Lemma~\ref{lemma_no_SO_code_via_weight_restrictions2} for $d=62$,
  and Corollary~\ref{cor_ex_non_existence_1} for singly-even $6\le d\le 26$, $66\le d\le 126$, and $d\in\{34,38\}$.
  Constructions for $d\le 16$ can be found in \cite{BouyuklievOstergaard2006Binary} and \cite{pless1972classification}.
  Constructions for $d\in\{38,40,54,56\}\cup\{66,\dots,92\}$ in Lemma~\ref{lemma_construction_aut}, see Appendix~\ref{sec_generator_matrices_of_SO_codes} for the corresponding generator matrices.
  For constructions with $d\in\{d\equiv 0\pmod 2\,:\, 18\le d\le 60\}\backslash\{38,42,46,54\}$ we refer to \cite{li2026nonexistence}.
  Generator matrices for $d\in\{42,46,94\}$ can be found in \cite{an2025shortest}.
  For $d\in\{72,76,84,86,88,98,\dots,106,110,114,118,122,126\}$ Corollary~\ref{cor_insert} gives examples.
  For doubly-even $d\ge 108$, $d\in\{64,96\}$ and for singly-even $d\ge 130$ there exist Griesmer
  codes, so that we can apply Lemma~\ref{lemma_griesmer_code_implies_SO}.
\end{proof}
For even $d\ge 2$ we have $n_2(8,d)=g_2(8,d)+3$ for $d\in\{26,\dots,32\}\cup\{60\}$, $n_2(8,d)=g_2(8,d)+2$ for $d\in\{10,\dots,20\}\cup\{34,\dots,58\}$, $n_2(8,d)=g_2(8,d)+1$ for $d\in\{4,6,8,22,24,62\}\cup\{68,\dots,92\}\cup\{100,\dots,104\}$, and $n_2(8,d)=g_2(8,d)$ otherwise.

\section{Optimal binary self-orthogonal codes with dimension $\mathbf{9}$ and minimum distance at most $\mathbf{64}$}
\label{sec_opt_9}

In this section we partially determine $\tilde{n}_2(9,d)$ for $d\le 64$, see Theorem~\ref{thm_SO_k_9} and Corollary~\ref{cor_SO_equals_BLC}. We remark that for even $d\le 64$ the exact values $n_2(9,d)$ are known if $d\le 28$ or $d\in\{32,40,50,52,64\}$. For the construction of self-orthogonal $9$-dimensional binary codes the methods from Section~\ref{sec_opt_at_most_seven} and Section~\ref{sec_opt_8} suffice for our purpose. This is only partially true for the presented non-existence results, so that we have to enhance our toolbox. As a by-product we improve the previously best known lower bound for $n_2(9,44)$ in Theorem~\ref{thm_no_93_9_44} and show divisibility results for some putative optimal $[n,9,d]_2$ codes in Proposition~\ref{prop_opt_implies_divisibility}.

For some set $\mathcal{F}\subset\N$ we say that $C$ is a linear code with forbidden weights $\mathcal{F}$ if $\wt(c)\notin\mathcal{F}$ for all $c\in C$. Let $y_r$ correspond to $y\in C$ in $\Res(C;x)$. From Equation~(\ref{eq_weight_sum}) and Equation~(\ref{eq_weight_residual}) we conclude
\begin{equation}
  \label{eq_weight_residual2}
  \wt\!\left(y_r\right)=\frac{\wt(y)+\wt(x+y)-\wt(x)}{2}=\frac{\wt(x)+\big(\wt(y)-\wt(x)\big)+\big(\wt(x+y)-\wt(x)\big)}{2},
\end{equation}
which directly implies:
\begin{lemma}
  \label{lemma_min_div}
  Let $C$ be a $2^r$-divisible binary code and $y_r$ correspond to $y\in C$ in $\Res(C;x)$. If $\wt(x)\equiv 2^r{\pmod {2^{r+1}}}$ and $\wt\!\left(y_r\right)\equiv 0{\pmod {2^r}}$, then either $\wt(y)$ or $\wt(x+y)$ is congruent to $2^r$ modulo $2^{r+1}$.
\end{lemma}

A codeword that is not even is called odd.
Applying Lemma~\ref{lemma_min_div} with $r=0$, Equation~(\ref{eq_weight_residual2}) implies the following improvement of the residual code argument from Lemma~\ref{lemma_residual_code_bound}.
\begin{lemma}
  \label{lemma_residual_code_bound_odd}
  Let $C$ be an $[n,k,d]_2$ code such that the weight of each odd codeword is at least $d'$. If $c\in C$ is a codeword with odd weight $w<2d$, then $\Res(C;c)$ is an $\left[n-w,k-1,\tfrac{d+d'-w}{2}\right]_2$ code.
\end{lemma}

Gaps in the weight spectrum can sometimes be exploited by a combination of Lemma~\ref{lemma_min_div} and Equation~(\ref{eq_weight_residual2}) to exclude further weights.
\begin{lemma}
  \label{lemma_no_weight_46_in_93_9_44}
  An even $[93,9,44]_2$ code does not contain a codeword of weight $46$.
\end{lemma}
\begin{proof}
  Let $C$ be an even $[93,9,44]_2$ code. The residual code argument in Lemma~\ref{lemma_residual_code_bound} gives that the weights of the non-zero codewords are contained in $$\{44,46,48,56,58,60,64,72,74,76,80,84,88,90,92\}.$$ Assume that $x$ is a codeword of weight $46$. Using \texttt{LinCode} we obtain that
  $\Res(C;x)$ is a $[47,8,21]_2$ code with weight distribution $21^{66}   22^{78}   23^{30}   24^{30}   29^{30}   30^{18}   31^{2}   32^{1}$.\footnote{Up to equivalence there is a unique $[47,8,21]_2$ code.} By applying Lemma~\ref{lemma_min_div} with $y_r\in \Res(C;x)$ of weight $32$ we can assume that $\wt(y)\equiv 2\pmod 4$. Using Equation~(\ref{eq_weight_residual2}) we compute
  \begin{center}
    \begin{tabular}{rrrrr}
      \hline
      $\wt(y)$   & 46 & 58 & 74 & 90 \\
      $\wt(x+y)$ & 64 & 52 & 36 & 20 \\
      \hline
    \end{tabular}
  \end{center}
  so that $\wt(x+y)=64$. In $\Res(C;x+y)$ the codewords $x,y\in C$ correspond to a residual codeword of weight $14$. However, using  \texttt{LinCode} we verify that a $[29,8,12]_2$ residual code $\Res(C;x+y)$ is doubly-even.\footnote{There are $73$ non-equivalent $[29,8,12]_2$ codes, which are all doubly-even and have a maximum weight of at most $28$.}
\end{proof}

Our next tool is a divisibility criterion for binary non-Griesmer codes, see Lemma~\ref{lemma_divisible_row_span}.
\begin{theorem} (\cite[Theorem 2.6]{dodunekov1985improvement}\footnote{In \cite{dodunekov1990comment} the statement was generalized to arbitrary field sizes, cf.\ \cite[Corollary 6]{Ward1998}.
})
  \label{thm_gen_mat_plus_t_q_2}
  If $C$ is an $[n,k,d]_2$-code with $n=g_2(k,d)+t$, then $C$ admits a generator matrix $G$ where every row has a weight between $d$ and $d+t$.
\end{theorem}
If $n=g_2(k,d)+t$ we speak of a \emph{Griesmer plus $t$ code and call $t$ the \emph{Griesmer defect}.}
\begin{lemma}
  \label{lemma_divisible_row_span}
  Let $C$ be an $[n,k,d]_2$-code that admits a generator matrix $G$ where every row has a Hamming weight that is divisible by $2^r$ with $r\ge 1$.
  If each residual code w.r.t.\ a row of $G$ is $2^{r-1}$-divisible, then $C$
  is $2^r$-divisible.
\end{lemma}
\begin{proof}
  For two codewords $x,y\in C$ and $y_r\in\Res(C;x)$ being the residual of $y$ we have
  $$
    \wt(x+y)=\wt(x)+\wt(y)-2\cdot\left|\supp(x)\cap \supp(y)\right|
  $$
  and
  $$
    \wt\!\left(y_r\right)=\wt(y)- \left|\supp(x)\cap \supp(y)\right|,
  $$
  see Equation~(\ref{eq_weight_sum}) and Equation~(\ref{eq_weight_residual}), respectively. For each row $x$ of $G$ the weight is divisible by $2^r$
  and $\wt\!\left(y_r\right)$ is divisible by $2^{r-1}$. So, if the weight of
  $y$ is divisible by $2^r$, then also the weight of $x+y$ is divisible by
  $2^r$. Since each non-zero codeword can be written as a sum of pairwise
  different rows of $G$ we can prove the statement by induction on the
  number of summands.
\end{proof}

\begin{proposition}\label{prop_opt_implies_divisibility} $\,$
  \begin{itemize}[noitemsep]
   \item Each $[33,7,14]_2$ code with forbidden weight $15$ is $2$-divisible.
   \item Each $[49,8,22]_2$ code with forbidden weight $23$ is $2$-divisible.
   \item Each $[61,8,28]_2$ code with forbidden weights $\{29,30,31\}$ is $4$-divisible.
   \item Each $[78,9,36]_2$ code is $4$-divisible.
   \item Each even $[93,9,44]_2$ code is $4$-divisible.
   \item Each $[117,9,56]_2$ code is $8$-divisible.
   \item Each $[125,9,60]_2$ code is $4$-divisible.
  \end{itemize}
 \end{proposition}
\begin{proof}
  For each $[n,k,d]_2$ code $C$ of the above list we compute the Griesmer defect $t=n-g_2(k,d)$, use the residual code argument in Lemma~\ref{lemma_residual_code_bound} and other tailored arguments to exclude some small weights. If the remaining values in $\{d,d+1,\dots,d+t\}$ are divisible by $2^r$ and the corresponding residual codes are $2^{r-1}$-divisible, then we can
  apply Theorem~\ref{thm_gen_mat_plus_t_q_2} and Lemma~\ref{lemma_divisible_row_span} to conclude that $C$ is $2^r$-divisible. For $r>1$ the divisibility of the residual codes is obtained from preceding parameters in the list, if not stated otherwise.
  \begin{itemize}
   \item A $[33,7,14]_2$ code is a Griesmer plus $3$ code. Since $C$ does not contain
   a codeword of weight $15$ the non-existence of a $[16,6,7]_2$ code and Lemma~\ref{lemma_residual_code_bound_odd} excludes weight $17$. Thus, $C$ is spanned by codewords of weight $14$ and $16$.
   \item A $[49,8,22]_2$ code is a Griesmer plus $2$ code. Since $C$ does not contain
   a codeword of weight $23$ it is spanned by codewords of weight $22$ and $24$.
   \item A $[61,8,28]_2$ code is a Griesmer plus $3$ code. The forbidden weights imply that $C$ is spanned by codewords of weight $28$. The corresponding residual codes are $[33,7,14]_2$ codes and Equation~(\ref{eq_weight_residual2}) implies that there is no codeword of weight $15$.
   \item A $[78,9,36]_2$ code is a Griesmer plus $2$ code with forbidden weights $\{37,38,39\}$, so that it is spanned by codewords of weight $36$. The corresponding residual codes are $[42,8,18]_2$ Griesmer plus $2$ codes and Equation~(\ref{eq_weight_residual2}) excludes weight $19$, i.e.\ they are spanned by even codewords.
   \item A $[93,9,44]_2$ code is a Griesmer plus $2$ code. The assumption that it is even and Lemma~\ref{lemma_no_weight_46_in_93_9_44} imply that it is spanned by codewords of weight $44$. The corresponding residual codes are $[49,8,22]_2$ codes and Equation~(\ref{eq_weight_residual2}) implies that there is no codeword of weight $23$.
   \item A $[117,9,56]_2$ code is a Griesmer plus $3$ code with forbidden weights $\{57,\dots,63\}$, so that it is spanned by codewords of weight $56$. The corresponding residual codes are $[61,8,28]_2$ codes and Equation~(\ref{eq_weight_residual2}) excludes the weights $\{29,30,31\}$.
   \item A $[125,9,60]_2$ code is a Griesmer plus $3$ code with forbidden weights
   $\{61,62,63,65,66\}$, so that it is spanned by codewords of weight $60$. The corresponding residual codes are $[65,8,30]_2$ Griesmer plus $3$ codes and Equation~(\ref{eq_weight_residual2}) excludes the weights $\{31,33\}$, i.e.\ they are spanned by even codewords.
  \end{itemize}
\end{proof}
The $4$-divisibility of $[78,9,36]_2$ codes is also shown in \cite[Lemma 3.6]{kurz2021codeunique} and the $8$-divisibility of $[117,9,56]_2$ codes in \cite[Theorem 3.10]{kurz2021codeunique}. In \cite[Theorem 3.3]{kurz2021codeunique} it was shown that there exists a unique $[46,9,20]_2$ code. Interestingly enough, it has a trivial automorphism group and is doubly-even. We will use the $4$-divisibility of even $[93,9,44]_2$ codes to show their non-existence in Theorem~\ref{thm_no_93_9_44}.

\begin{lemma}
  \label{lemma_unique_99_8_48}
  There is a unique $[99,8,48]_2$ code. It has weight distribution
  $48^{210} 56^{36} 64^{9}$ and an automorphism group of order $144$. A generator matrix is given by
  $$
    \left(\begin{smallmatrix}
111000000001111111111110111011111111111100011000000001111110010111110110100000011111100000001110000\\
000111100001111111100001000111111111111000000111100001100001101111001011000000011110011111111110011\\
000000011110000000011111000011111110000010011110011000011111100111000010010000011110000111001111111\\
111111111111111111111111000000000000000001011110000110000000011111111010001000011111111111000111100\\
000111111110000111111110111000011110000000100111111111100111110100010110000100000001100000000110000\\
011001100110011001100111011001100110011111101010101011101110100001011010000010000111111001110111111\\
001010101010101010101010001110101010101100110011010010101010110110111110000001000110000111110001100\\
110010001101001011010110010000101111100110000011101110111001011011001100000000101010101011010010101\\
    \end{smallmatrix}\right).
  $$
\end{lemma}
\begin{proof}
  Let $C$ be a $[99,8,48]_2$ code. Since no $[98,8,48]_2$ code exists it has full length. The Griesmer defect of $C$ is $2$ and the residual code argument in Lemma~\ref{lemma_residual_code_bound} excludes all odd weights except $\{49,97,99\}$. The residual code of a codeword of weight $49$ is a $[50,7,24]_2$ code. It is unique and has weight distribution $24^{108} 32^{19}$, see e.g.\ \cite{kurz2020classification} where a precursor of \texttt{LinCode} was used for the exhaustive enumeration.\footnote{Due to the existence of a unique $[51,8,24]_2$ projective two-weight code this classification result is not contained in \cite{bouyukliev2001optimal}.} With this we compute the unique weight distribution $48^{108} 49^{128} 64^{19}$. Using \texttt{LinCode} we have exhaustively enumerated all possibilities for the even-weight subcode. It is a unique $[96,7,48]_2$ code with weight distribution $48^{124} 64^3$
  -- contradiction.\footnote{Alternatively, we can use the MacWilliams equations to conclude $B_{98}=-1<0$.} Thus, $C$ is spanned by codewords of weight $48$ and $50$, so that Lemma~\ref{lemma_divisible_row_span} implies that $C$ is even.

  From the residual code argument in Lemma~\ref{lemma_residual_code_bound} we conclude that the non-zero weights are contained in $\{48,50,52,56,64,72,80,96,98\}$. Assume that $c_{50}$ is a codeword of weight $50$. Its residual code is a $[49,7,23]_2$ code, which we exhaustively enumerate using \texttt{LinCode}. There are exactly two non-equivalent such codes and they have weight distributions $23^{52} 24^{56} 31^{12} 32^{7}$ or $23^{48} 24^{60} 31^{16} 32^3$, respectively. Thus, the non-zero weights of $C$ have to be contained in $\{48,50,56,64\}$. Using \texttt{LinCode} we have enumerated those $[99,8,48]_2$ codes. The unique such code has weight distribution $48^{210} 56^{36} 64^{9}$, i.e.\ no codeword of weight $50$ exists. Assume that $c_{98}$ is a codeword of weight $98$. Since $\Res\!\left(C;c_{98}\right)$ is a $[1,1,1]_2$ code,
  $C$ shortened onto $\supp(c_{98})$ is a $[98,7,48]_2$ code containing a codeword of weight $98$, which is impossible (using the non-existence of a codeword of weight $50$). Thus, $C$ is $4$-divisible.

  Assume that $c_{52}$ is a codeword of weight $52$. Its residual code is a $[47,7,22]_2$ code, which we exhaustively enumerate using \texttt{LinCode}. There is just one such code and it has weight distribution $22^{78} 24^{30} 30^{18} 32^{1}$. Thus, the non-zero weights of $C$ have to be contained in $\{48,52,56,64\}$. Using \texttt{LinCode} we have enumerated those $[99,8,48]_2$ codes. Again, the unique such code has weight distribution $48^{210} 56^{36} 64^{9}$, i.e.\ no codeword of weight $52$ exists. Thus, $C$ is $8$-divisible.

  Since the residual codes are $4$-divisible, there cannot be a codeword of weight $96$, whose residual code would have effective length $3$.
  Using $\texttt{LinCode}$ we have enumerated all $8$-divisible $[99,8,48]_2$ codes with maximum weight at most $80$.\footnote{We remark that there are unique residual codes for putative codewords of weight $72$ and $80$, which e.g.\ imply $A_{72}+A_{80}\le 1$.}
\end{proof}
We remark that the $[99,8,48]_2$ code in Lemma~\ref{lemma_unique_99_8_48} cannot be extended to a $[101,9,48]_2$ code. So, if a $[101,9,48]_2$ code exists, it has to be projective.
For the multiset of points $\cM$ corresponding to the unique $[99,8,48]_2$ code we have $\cM(P)\le 1$ for every point $P$ 
and $\cM(E)\le 6$ for every plane $E$, so that we cannot construct a $[92,8,44]_2$ code via $\cM-\chi_E$.

\begin{lemma}
  \label{lemma_no_99_9_46_SO}
  There exists no $[99,9,46]_2$ SO code.
\end{lemma}
\begin{proof}
  Let $C$ be a $[99,9,46]_2$ SO code. In Lemma~\ref{lemma_unique_99_8_48}
  we have classified the possible subcode of doubly-even codewords.
  Via a \texttt{LinCode} computation we verify that no extension to a $[99,9,46]_2$ code exists.
\end{proof}

\bigskip

Using similar methods as in Lemma~\ref{lemma_no_121_8_58_SO} we obtain the following non-existence result.
\begin{lemma}
  \label{lemma_no_130_9_62}
  There exists no $[130,9,62]_2$ SO code.
\end{lemma}
\begin{proof}
  By Corollary~\ref{cor_ex_non_existence_2}, $C$ has full length.
  The residual code argument in Lemma~\ref{lemma_residual_SO_code_bound} excludes all weights in $\{66,70,74,86,90,98,102,106,114,118,122\}$. By $C_4$ we denote the subcode of the doubly-even weight codewords. A $[128,8,64]_2$ code is a Griesmer code with weight distribution $64^{254} 128^1$. For a full-length doubly-even $[129,8,64]_2$ code the residual code argument in Lemma~\ref{lemma_residual_code_bound} excludes all non-zero weights except $\{64,128\}$, which is impossible due to the MacWilliams identities.\footnote{More precisely, the MacWilliams identities give $A_{64}=252$, $A_{128}=3$, and $B_3=-2603<0$.} Now assume that
  $C_4$ is a full-length doubly-even $[130,8,64]_2$ code. The residual code argument in Lemma~\ref{lemma_residual_code_bound} excludes all non-zero weights except $\{64,68,80,128\}$. If $c\in C_4$ is a codeword of weight $128$, then the residual
  code is a $[2,1,2]_2$ code and $C'$ arising from $C_4$ by restricting to $\supp(c)$
  is a full length $[128,8,62]_2$ code with weight distribution $62^{64} 64^{126} 66^{64} 128^1$\footnote{Such a code indeed exists and it has an automorphism group of order $1981808640$.}, so that $C_4$ has weight distribution $64^{190} 68^{64} 128^1$. If there is no codeword of weight $128$ in $C_4$, but a codeword of weight $80$, then \texttt{LinCode} shows that the residual code is a unique $[50,7,24]_2$ code with weight distribution $24^{108} 32^{19}$, cf.~\cite{kurz2020classification}. Counting yields the unique weight distribution $64^{235} 80^{20}$ for $C_4$. If the maximum weight of $C_4$ is $68$, then the residual code is a $[62,7,30]_2$ code. \texttt{LinCode} shows that there are exactly two such non-equivalent codes, both with a codeword of weight $62$, which yields a contradiction in our setting. Thus, we have classified exactly three possible weight distributions for $C_4$.\footnote{Clearly, $C_4$ cannot consist of non-zero codewords of weight $64$ only, since e.g.\ the MacWilliams equations give a contradiction.}

  Since no $[130,9,64]_2$ code exists, there is a codeword $c\in C$ with weight $62$. The residual code $\Res(C;c)$ is a $[68,8,32]_2$ code. Using \texttt{LinCode} we exhaustively enumerate all $153$ such codes, where all of them are even automatically. Having the possible weight distributions for the doubly-even codewords and the residual codes of a codeword of weight $62$ at hand, we use the split weight enumerator with the condition $B_2(C) \le B_2\!\left(C_4\right)$ to obtain a contradiction.\footnote{Alternatively, we may verify that for each possible weight distribution of $C_4$ the first three MacWilliams equations complemented by $B_2(C) \le B_2\!\left(C_4\right)$ do not have a solution.}
\end{proof}

\todo{Describe the $[130,8,64]_2$ code with weight distributions $64^{190} 68^{64} 128^1$ and $64^{235} 80^{20}$ geometrically.}

\begin{theorem}
  \label{thm_no_93_9_44}
  There is no $[93,9,44]_2$ code.
\end{theorem}
\begin{proof}
  Assume that $C$ is an even $[93,9,44]_2$ code\footnote{If $C$ is a $[93,9,44]_2$ code that is not even, then puncturing and adding a parity bit gives an even $[93,9,44]_2$ code.} and $\cM$ be its corresponding multiset of points. Since no $[91,8,44]_2$ code exists $C$ is projective. From Proposition~\ref{prop_opt_implies_divisibility} we conclude that $C$ is doubly-even. Thus, there cannot be a codeword of weight $92$ whose residual would be $2$-divisible having effective length $1$. If $c_{88}$ is a codeword of weight $88$, then $C$ shortened onto $\supp(c_{88})$ is a $[88,k',44]_2$ code with non-zero weights in $\{44,88\}$, so that $k'\le 4$, see e.g.\ \cite{Bonisoli1984}. However, $\Res(C;c_{88})$ is an even $[5,k'',2]_2$ code, where $k''\le 4$, which contradicts $k'+k''=9$. The residual code argument in Lemma~\ref{lemma_residual_code_bound} excludes the weights $\{52,68\}$.
  Using \texttt{LinCode} we exhaustively enumerate all two non-equivalent $[33,8,14]_2$ and all $73$ non-equivalent $[29,8,12]_2$ codes, which are the residual codes of codewords of weight $60$ and $64$, respectively. In both cases applying the split weight enumerator with respect to the support of such a codeword and its complement yields a contradiction for each possible weight distribution of the residual code. Thus, $C$ does not contain codewords of weight $60$ or $64$.
  Using these constraints, the integer programming method yields $B_3>0$,\footnote{The rational point $A_{44}=\tfrac{2224}{7}$, $A_{48}=131$, $A_{56}=59$, $A_{72}=\tfrac{23}{7}$ satisfies $B_1 = B_2 = B_3 = 0$, i.e.\ we need integer rounding cuts.} i.e.\ there exists a line $L$ where all three points have multiplicity one in $\cM$. With this, the projection $\cM_L$ corresponds to a $[90,7,44]_2$ code $C'$. Using \texttt{LinCode} we exhaustively enumerate all six non-equivalent  $[90,7,44]_2$ codes. Since each of them contains either a codeword of weight $60$ or a codeword of weight $64$, we obtain a contradiction.
\end{proof}

\begin{theorem}
  \label{thm_SO_k_9}
  For even $2\le d\le 64$ we have the following bounds for $\tilde{n}_2(9,d)$:
  \begin{center}
    \begin{tabular}{lcccccccccccccccc}
    \hline
    $d$                & 2 & 4 & 6 & 8 & 10 & 12 & 14 & 16 & 18 & 20 & 22 & 24 & 26 & 28 & 30 & 32 \\
    $\tilde{n}_2(9,d)$ & 18 & 18 & 20 & 21 & 29 & 30 & 36 & 38 & 45 & 46 & 53 & 54 & 61 & 62 & 68 & 70 \\
    \hline
    \end{tabular}
    \begin{tabular}{lcccccccccccccccc}
    \hline
    $d$ & 34 & 36 & 38 & 40 & 42 & 44 & 46 & 48 & 50\\
    $\tilde{n}_2(9,d)$ & 77 & 78--79 & 84 & 86 & 92--93 & 94--95 & 100 & 101--102 & 108--109\\
    \hline
    \end{tabular}
    \begin{tabular}{lcccccccccccccccc}
    \hline
    $d$ & 52 & 54 & 56 & 58 & 60 & 62 & 64 \\
    $\tilde{n}_2(9,d)$ & 110 & 115--116 & 117--118 & 124 & 125--126 & 131 & 133 \\
    \hline
    \end{tabular}
  \end{center}
\end{theorem}
\begin{proof}
  Assume $d\ge 4$. Besides the Griesmer SO bound, the lower bounds are given by
  $n\ge 2k$ for $d=4$, $\tilde{n}_2(9,d)\ge n_2(9,d)$\footnote{Besides the data in \cite{grassl_codetables} we use the non-existence of a $[85,9,40]_2$ code \cite[Theorem 3.7]{kurz2021codeunique} and the non-existence of a $[93,9,44]_2$ code shown in Theorem~\ref{thm_no_93_9_44}.} for doubly-even $8\le d\le 64$,
  $\tilde{n}_2(9,d)\ge n_2(8,d+2)$ for singly-even $d\in\{6,\dots,58\}\backslash\{46\}$,
  Lemma~\ref{lemma_no_99_9_46_SO} for $d=46$,
  Lemma~\ref{lemma_no_130_9_62} for $d=62$, and
  Lemma~\ref{lemma_SO_non_existence_LinCode} or \cite{shimizu2026selforthogonal} for $d=22$. Constructions are given for doubly-even $8\le d\le 64$ in \cite{MagmaBKLC}, for $d\in\{2,4,6,10,14\}$ in \cite{BouyuklievOstergaard2006Binary}, for $d\in\{22,26,50\}$ in \cite{Lim2026ConstraintPreserving}, for $d\in\{18,34\}$ in Appendix~\ref{sec_generator_matrices_of_SO_codes}, and for $d\in\{30,38,42,46,54,56,58,62\}$ in Lemma~\ref{lemma_construction_aut}.
\end{proof}
The values $\tilde{n}_2(9,d)$ were previously determined in \cite{BouyuklievOstergaard2006Binary} for all $d\le 16$. A $[84,9,38]_2$ SO code was also found in \cite{Lim2026ConstraintPreserving}. To set the open cases of $\tilde{n}_2(9,d)$ into context, we remark that for even $d\le 64$ the exact values $n_2(9,d)$ are known if $d\le 28$ or $d\in\{32,40,50,52,64\}$. The $[92,8,44]_2$ code stored in the \texttt{BKLC} database cannot be extended to a $[92,9,42]_2$ SO code. Currently no other doubly-even $[92,8,44]_2$ code is known.
The fact that each doubly-even code is self-orthogonal, the bounds in Theorem~\ref{thm_SO_k_9}, and the divisibility results in Proposition~\ref{prop_opt_implies_divisibility} imply:
\begin{corollary}
  \label{cor_SO_equals_BLC}
  We have $\tilde{n}_2(9,d)=n_2(9,d)$ for $d\in \{36,56,60\}$.
\end{corollary}
\todo{Obtain divisibility results for putative $[101,9,48]_2$ and $[94,9,44]_2$ codes.}

\section*{Declaration on generative AI}
In short, generative AI was used strictly for copy-editing, correcting grammatical errors, and identifying inconsistencies in the notation, without altering the mathematical content. Everything else, including programming and setting up the computations, was done by the author himself. The models used were the free versions of GPT-5.6 Luna (OpenAI) and Claude Sonnet 5.5 (Anthropic).

On October 8, 2026, Hirotaka Shimizu contacted me via email to inform me that he had found a computer-assisted proof suggesting that the maximum possible minimum distance of a self-orthogonal $[52,9]_2$ code is $20$. A written summary and a verification in Lean 4 were provided. However, an intermediate step relied on one of my computer enumerations \cite{kurz2020classification,Kurz2026Codes}, where (a precursor of) \texttt{LinCode} \cite{bouyukliev2021computer} was used. The next day, I sent him a direct \texttt{LinCode} enumeration that verifies his result. Moreover, I mentioned that, although I had not previously worked on SO codes, I expected that the possible minimum distances of binary SO codes in dimension~$8$, as well as those in dimension~$9$ for \textit{small} values of $d$, could be completely determined using standard techniques and some effort. Still slightly taken aback by OpenAI's math drop, I started asking myself what properties a useful research paper should have. The present paper arose in that context. During the process of writing the paper, I realized how cumbersome and predictable some of the steps involved are.
Generative AI can be really useful for simplifying some of the arguments and extending the range of parameters that can be handled, cf.~Footnote~\ref{fn_AI}. I have decided not to incorporate any of these fascinating technical possibilities into this paper, but I intend to work with Hirotaka Shimizu to develop self-contained and simplified proofs of the stated results and extensions of these results that do not rely on extensive computer enumerations. To this end, we intend to make extensive use of generative AI.


\appendix

\section{Generator matrices of self-orthogonal codes}
\label{sec_generator_matrices_of_SO_codes}
In this section we list generator matrices for $SO$ codes found by computer searches. The examples of $4$-divisible non-projective SO codes are found by prescribing a subgroup of the automorphism group and solving an integer linear programming formulation. The singly-even SO codes were found by starting from some candidate for the doubly-even weight subcode and extension using \texttt{LinCode} \cite{bouyukliev2021computer}. Generator matrices with length at least $120$ are stated in octal format after appending them with zero
columns on the right to ensure that the length is a multiple of three.
An $[n,k,d]_q$ code is called \emph{minimal} if the support of every non-zero codeword is not strictly contained in the support of another codeword.
The Ashikhmin--Barg condition $\tfrac{d}{w_{\max}}>\tfrac{q-1}{q}$ implies minimality \cite{Ashikhmin1998Minimal}, where $w_{\max}$ denotes the maximum weight. It is satisfied for several of the following examples and can be easily read off from the stated weight distribution. For examples of minimal self-orthogonal codes violating the Ashikhmin--Barg condition we refer e.g.\ to \cite{jin2025several}.

A $[82,8,38]_2$ SO code is given by
$$
  \left(\begin{smallmatrix}
1110000000001111000000000111111110000000111111111111111100100101100011100011000000\\
0001111000000000111110000000000001111111111111111111111100000000011111100010100000\\
0000000110000000000001111111111111111111000000011111111100010011111100100100010000\\
0000000001111111111111111000011110001111000111100001111100001110100101000000001000\\
0111111110110011011110011001100110110011011001100110000111111100001110111110000100\\
1110011111110111000111101110111010000100100000100110011000000000011111011100000010\\
1010101111011101101010111011100101110100100011001010101111000000100110000010000001\\
1000111011101000011001001010001110011101011100100101001001111011011110001000000000\\
  \end{smallmatrix}
  \right).
$$
It has weight distribution $38^{83}   40^{108}   46^{42}   48^{17}   54^{3}   56^{2}$ and a trivial automorphism group.

A non-projective $4$-divisible $[84,8,40]_2$ code is given by
$$
  \left(\begin{smallmatrix}
000000000000000000000000000000000000111111111111111111111111111111111111111111111111\\
000000000000000000001111111111111111000000000000000000000000111111111111111111111111\\
000000000000111111110000000011111111000000000000111111111111000000000000111111111111\\
000000111111000011110000001100001111000000111111000011111111000000111111000000001111\\
000111001111001100010011110101110011000011000111000100001111000001000011000011110111\\
001001010011010001101100111000011100001101111001011000110001001110011101011100010001\\
011110000101101110110011011010100101010100001010101100000010010011100110100101101010\\
110010000011000110100101110110100010101010011110111001000100110110001011100011100001\\
  \end{smallmatrix}\right).
$$
It has weight distribution $40^{192} 48^{57} 56^6$ and an automorphism group of order $6$.

A $[113,8,54]_2$ SO code is given by
$$
  \left(\begin{smallmatrix}
11000000110000111100001111111100111100000000000011110000000011111111000000001111111111111111001001100111110000000\\
00110000001100000011111111111100000011110000000000001111000000000000111111111111111111111111000100011111101000000\\
00001100000011111111110000111100000000001111000000000000111111111111111111110000000011111111000011111001100100000\\
00000011111111001100110011001100000000000000111111111111111100001111000011110000111100001111001110101010100010000\\
00000000000000000000000000000010111111111111111111111111111111111111111111111111111111111111110000000000000001000\\
11111111111111111111111111111101001100110011001100110011001100110011001100110011001100110011110000000000000000100\\
01010101010101010101010101010111010101010101010101010101010101010101010101010101010101010101011111010100100000010\\
01011100101100101111100001011000110101000101001100101101100101101100100011101010001110110100111010000110100000001\\
  \end{smallmatrix}
  \right).
$$
It has weight distribution $54^{91}   56^{120}   62^{28}   64^{7}   70^{7}   78^{2}$ and an automorphism group of order $21$.

A non-projective $4$-divisible $[115,8,56]_2$ code is given by
$$
  \left(\begin{smallmatrix}
1000010001000000111100001111111111101111111111111111010000001111000111110011111111000011111111111111111111111111110\\
0100001000111100000011111111111111000000000000111111001000000000100000000000000000111111111111111111000011111111110\\
0010000100000011111111110000111111000000001111000011000100000000010000001111111111111100000000111111111100001111110\\
0001000011111111001100110011001111000011110111001100000010000000001111111100001111001100001111000011000100110011110\\
0000111110001101010101110101010111001100110001110000100000000011111001110100110011010000110011001100011100111100110\\
1111001001010101100101010110010111011100110110000011000001001100011110001011000011010100111100001100011011110011000\\
1111101111000011111011111111000001110101011011010101000000100011000000110000000011000000000000111111000011110011111\\
0110000011100111001101100010011001101110011000110110000000010101111010010001010101100101010101010101101101010101011\\
  \end{smallmatrix}\right).
$$
It has weight distribution $56^{210} 64^{36} 72^8 80^1$ and an automorphism group of order $42$.

A non-projective $4$-divisible $[124,8,60]_2$ code is given by
$$
  \left(\begin{smallmatrix}
777777777777777777777777777777777777777774\\
000000000000000000007777777777777777777774\\
000000000077777777770000000000777777777774\\
000007777700000777770000077777000007777774\\
003770037700177001770017700177003770001774\\
074170741703617036170360703607034170376074\\
314631146314663146231463114633144631406314\\
525253252525225252452525325255251252572524\\
  \end{smallmatrix}\right).
$$
It has weight distribution $60^{127} 64^{127} 124^1$ and an automorphism group of order $332943851520$. It can be obtained by adding the all-one codeword to the $[124,7,60]_2$ non-projective two-weight code obtained from Lemma~\ref{lemma_non_projective_2wt_2} with $k=7$.

A non-projective $4$-divisible $[140,8,68]_2$ code is given by
$$
  \left(\begin{smallmatrix}
00000000000000000000037777777777777777777777776\\
00000000001777777777740000000000037777777777776\\
00000377776000007777740000037777740000037777776\\
00177400776003770017740017740017740017740017776\\
03607417036074170360741760740360740760340760076\\
14631463146314631463146163143463143163443161736\\
74007417740077607400740777055400777020744666706\\
25252525252525252525252265264525245265511322112\\
  \end{smallmatrix}\right).
$$
It has weight distribution $68^{126}   72^{126}   76^2   128^1$ and an automorphism group of order $1152$.

A non-projective $4$-divisible $[147,8,72]_2$ code is given by
$$
  \left(\begin{smallmatrix}
0000000000000000000000377777777777777777777777777\\
0000000000017777777777400000000000017777777777777\\
0000003777760000077777400000077777760000003777777\\
0017774003760017700077400077700377760003774000777\\
0360174074160363701703403707703401760774074037017\\
3761614100623564702115474700114476361034104347377\\
5422224634260673116066415311724316467154700473063\\
4446461305625435226334125522361123522265331155125\\
  \end{smallmatrix}\right).
$$
It has weight distribution $72^{210} 80^{36} 88^6 96^3$ and an automorphism group of order $18$.

A non-projective $4$-divisible $[155,8,76]_2$ code is given by
$$
  \left(\begin{smallmatrix}
0000000000000000000000000777777777777777777777777776\\
0000000000377777777777777000000000017777777777777776\\
0000017777400000017777777000003777760000000177777776\\
0017760377400377760017777007774037760007777600077776\\
0360761403403401760760377170374140360170017607700776\\
1463163475435476167161437273435643563630363631717036\\
6421761714006516217202157031542705674651465612361046\\
3557224526557626661464663756657216425372536664526312\\
  \end{smallmatrix}\right).
$$
It has weight distribution $76^{204} 80^{31} 92^{19} 108^1$ and an automorphism group of order $6$.


A non-projective $4$-divisible $[171,8,84]_2$ code is given by
$$
  \left(\begin{smallmatrix}
000000000000000000000000077777777777777777777777777777777\\
000000000000077777777777700000000000000007777777777777777\\
000000077777700000077777700000000777777770000000077777777\\
000077700077700077700077700007777000077770000777700007777\\
017703701701700703701701701770077007701770017003701770077\\
060714306316307134316306316070707070716070363074316070707\\
363101713766411754466011766311133113366311477314466311133\\
525224520533123325522522532533311231122722525525522532215\\
  \end{smallmatrix}\right).
$$
It has weight distribution $84^{186} 88^{30}   92^{37}   96^{1}   124^{1}$ and an automorphism group of order $240$.

A non-projective $4$-divisible $[186,8,92]_2$ code is given by
$$
  \left(\begin{smallmatrix}
00000000000000000000000000000077777777777777777777777777777777\\
00000000000000037777777777777700000000000000007777777777777777\\
00000001777777740000000777777700000001777777770000000037777777\\
00007776000177740007777000177700017776000077770000777740003777\\
00770076017603740370077007603700760176007700770177003740374077\\
17171716363634741410303030604317163636371717170607014141414103\\
63030366074754147461417141635463020746076363031630060746074714\\
25252522525265252522525252254525245252512525253253125252525325\\
  \end{smallmatrix}
  \right).
$$
It has weight distribution $92^{189} 96^{63} 124^3$ and an automorphism group of order $16515072$. The three codewords of weight $124$ form a $62$-fold repetition of the $[3,2,2]_2$ simplex code, i.e.\ they have full support. Removing these codewords, i.e.\ taking a suitable largest subcode not containing them, gives a non-projective two-weight code, see e.g.\ \cite{Kurz2024NonProjective},  with weight distribution $92^{48} 96^{15}$ and an automorphism group of order $4608$. A generator matrix is given by
$$
  \left(\begin{smallmatrix}
00000000000000000000000000000077777777777777777777777777777777\\
00007776000177740007777000177700017776000077770000777740003777\\
00770076017603740370077007603700760176007700770177003740374077\\
17171716363634741410303030604317163636371717170607014141414103\\
63030366074754147461417141635463020746076363031630060746074714\\
25252522525265252522525252254525245252512525253253125252525325\\
   \end{smallmatrix}\right)
$$
Let the multiset $\cM$ correspond to the shortened full-length $[183,6,92]_2$ code. Writing $V$ for the ambient space $\PG(5,2)$ and using two suitable disjoint sets of points $\cM_1$, $\cM_2$, we can write $\cM=3\cdot \chi_V-\cM_1+\cM_2$, so that we have $\cM(P)=2$ iff $\cM_1(P)=1$ and $\cM(P)=4$ iff $\cM_2(P)=1$. The set of points $\cM_1+\cM_2$ corresponds to a projective $[12,5,4]_2$ SO code with weight distribution $4^6 6^{16} 8^9$ and an automorphism group of order $1152$. A generator matrix is given by
$$
  \left(\begin{smallmatrix}
000000111111\\
011011000110\\
100111001001\\
111001011011\\
001100110011\\
  \end{smallmatrix}\right).
$$
We remark $\tilde{n}_2(5,4)=11$ and that there are five non-equivalent $[12,5,4]_2$ SO codes \cite{BouyuklievOstergaard2006Binary}.
\todo{Describe this code geometrically.}

\bigskip
\bigskip

A $[45,9,18]_2$ SO code is given by
$$
  \left(\begin{smallmatrix}
111111111000000111111000100011111110010000000\\
111111000111000111000111010011111001101000000\\
000000000000111111111111001000011111100100000\\
111111111111111111000000000100000111100010000\\
000111001001011000011011111100101010100001000\\
001011010010101011001001001111101011000000100\\
000101100110010101010000111001110111100000010\\
110110010011110001100101100111010010000000001\\
010011111011010111110110111110011011100000000\\
  \end{smallmatrix}
  \right).
$$
It has weight distribution $18^{60} 20^{135} 22^{115} 24^{79} 26^{71} 28^{37} 30^{9}   32^4 34^1$ and a trivial automorphism group.

A non-projective $4$-divisible $[70,9,32]_2$ code is given by
$$
  \left(\begin{smallmatrix}
0100000000110010001011100110100110001100000001001111101011010111011101\\
0010010000000100110011100110111110100000010101101100010101100011110011\\
1001010001110000000010010001111011111110000100001011010011000101110011\\
0000110001000110111110110000011010000101010010011101010010001101011101\\
0000000100010110010111000100010001010111101110100110010110101100110011\\
0000001011100110011000110111101110100111101110010110011001011111110011\\
0000000000001110101011010011110111101111011110000101111011100000000000\\
1000001000010111101111100010011000000000101011010010010101010111101110\\
1000001000010110110000011100010110110111001000101101100001000111101110\\
  \end{smallmatrix}
  \right).
$$
It has weight distribution $32^{318} 40^{192} 64^1$ and an automorphism group of order $774144$.

A $[77,9,34]_2$ SO code is given by
$$
  \left(\begin{smallmatrix}
00000000000000000000000000000011111111111111111111111111111111111111111110001\\
00000000000000001111111111111100000000000000000000011111111111111111111110000\\
00001111111111110000011111111100000011111111111111100000000000111111111110010\\
00010000001111110011100000111100011100000001111111100000011111000001111110000\\
00110000110001110100100011001100101100011110000011100001100011011110011110100\\
01110111110110010000000101000101100000100110000100101110101101101111100110000\\
11000011111110001011001001001001010101001010011111100011101100110010011111000\\
11101101010010110101000010010011101001110000101001000101110111000110101010000\\
11111011000001011000010110011100001101100100110101010110011011000000110000000\\
  \end{smallmatrix}
  \right).
$$
It has weight distribution $34^{97} 36^{147} 38^{67} 40^{61} 42^{61} 44^{37} 46^{25}   48^6 50^6 52^4$ and a trivial automorphism group. The doubly-even $[77,8,36]_2$ subcode has an automorphism group of order $6$. We remark that there exist more than $62000$ optimal $[77,8,36]_2$ codes, but we do not know an example with a larger automorphism group.


A non-projective $4$-divisible $[118,9,56]_2$ code is given by
$$
  \left(\begin{smallmatrix}
1000000000000000000001111110000000110000000111111111111100000001111111111110000000111111111111100000001100000001111111\\
0110011000000000000001011101101001101100110111000010110111010011111001010101100110100000011100110101011110101011010000\\
0001111000000000000000111111011010010111100011100101010010110100010111010100111100001100010001101100110101100110011110\\
0000000101010100000000010110110011110110011000010011010010110101110111101111101001001011001000010101011110110101101000\\
0000000011001100000000110010001111100001111101101011101010101010110000001101011010010111010110101100110110101011001100\\
0000000000111100000001100001100110001100110001001011001101100110000101011111010101001100101101100011111101100111111001\\
0000000000000010101011110101011010101101001011111010100001100110001100110100110011111101100010010110101010101010011010\\
0000000000000001100110000001010101101011010000111101011000011111011001010100001111101111011110010101010001100111011101\\
0000000000000000011110111000110011001010101110110101101011001101110010111101100110100101111001001100111000011110100000\\
  \end{smallmatrix}\right)
$$
It has weight distribution $56^{348} 64^{127} 72^{36}$ and an automorphism group of order $384$.

\end{document}